\documentclass{amsart}
\usepackage[dvipsnames]{xcolor}
\usepackage{amssymb}
\usepackage{mathrsfs}
\usepackage{tikz-cd}
\usepackage{graphicx}
\usepackage{mathtools}
\usepackage{hyperref}
\usepackage{array}
\usepackage{microtype}
\usepackage{pinlabel}

	\newtheorem{theorem}[subsection]{Theorem}
	\newtheorem{corollary}[subsection]{Corollary}
	\newtheorem{proposition}[subsection]{Proposition}
	\newtheorem{lemma}[subsection]{Lemma}
\theoremstyle{definition}
	\newtheorem{construction}[subsection]{Construction}
	\newtheorem{notation}[subsection]{Notation}
	\newtheorem{example}[subsection]{Example}
	\newtheorem{definition}[subsection]{Definition}
\theoremstyle{remark}
	\newtheorem{remark}[subsection]{Remark}

\DeclareFontFamily{U}{matha}{}
\DeclareFontShape{U}{matha}{m}{n}{
  <-5.5>    matha5
  <5.5-6.5> matha6 
  <6.5-7.5> matha7
  <7.5-8.5> matha8
  <8.5-9.5> matha9
  <9.5-11>  matha10
  <11->     matha12
}{}
\DeclareSymbolFont{matha}{U}{matha}{m}{n}
\DeclareFontSubstitution{U}{matha}{m}{n}
\DeclareFontFamily{U}{mathx}{\hyphenchar\font45}
\DeclareFontShape{U}{mathx}{m}{n}{<-> mathx10}{}
\DeclareSymbolFont{mathx}{U}{mathx}{m}{n}
\DeclareFontSubstitution{U}{mathx}{m}{n}

\DeclareMathDelimiter{\ldbrack}{4}{matha}{"76}{mathx}{"30}
\DeclareMathDelimiter{\rdbrack}{5}{matha}{"77}{mathx}{"38}

\newcommand{\initialE}{\varnothing}
\renewcommand{\emptyset}{\varnothing}

\DeclareMathOperator{\id}{id}

\newcommand{\shortdot}{{\cdot}}

\newcommand{\coloredopfont}{\mathscr} 
\newcommand{\CoCo}{\mathfrak{CoCo}}
\newcommand{\as}{\mathscr{A}\mkern-3mu{s}}

\newcommand{\underlying}[1]{\lvert #1\rvert}

\newcommand{\liststyle}{\underline}

\newcommand{\freeq}{\mathcal{F}_{{\calQ}}}
\newcommand{\colorsa}{\mathsf{A}}
\newcommand{\colorsb}{\mathsf{B}}
\newcommand{\colorsc}{\mathsf{C}}
\newcommand{\colorsd}{\mathsf{D}}
\newcommand{\algebras}[1]{#1\text{-}\mathrm{alg}}
\newcommand{\bimod}[2]{#1\text{-}\mathrm{mod}\text{-}#2}

\newcommand{\bv}{\mathcal{BV}}
\newcommand{\ger}{\mathcal{G}}

\newcommand{\twohoms}[2]{[#1, #2]}
\newcommand{\threehoms}[3]{\ldbrack #2, #3 \rdbrack}

\newcommand{\initial}[1]{\mathring{#1}}

\newcommand{\calQ}{\mkern-1.25mu\mathcal{Q}}
\newcommand{\initialQ}{\mkern-1.25mu\initial{\mathcal{Q}}}

\newcommand{\cp}{\mathbin{\mkern-.5mu\triangleleft}}
\newcommand{\vprof}[2]{\tbinom{#1}{#2}}

\newcommand{\oprm}{\mathrm{op}}
\newcommand{\ua}{\liststyle{a}}
\newcommand{\ub}{\liststyle{b}}
\newcommand{\uc}{\liststyle{c}}

\newcommand{\uu}{\liststyle{u}}
\newcommand{\uv}{\liststyle{v}}
\newcommand{\uw}{\liststyle{w}}
\newcommand{\listsopa}{\mathsf{S}_{\colorsa}^{\oprm}}
\newcommand{\listsopb}{\mathsf{S}_{\colorsb}^{\oprm}}

\newcommand{\listsa}{\mathsf{S}_{\colorsa}}
\newcommand{\listsb}{\mathsf{S}_{\colorsb}}
\newcommand{\listsc}{\mathsf{S}_{\colorsc}}

\newcommand{\obj}{\mathrm{obj}}
\newcommand{\alg}{\mathrm{alg}}

\newcommand{\ml}{\smash{\mathrm{mul}}\vphantom{non}}
\newcommand{\notml}{\mathrm{non}}

\DeclareMathOperator{\nbhd}{nb}

\title{Coextension of scalars in operad theory}
\author[G.~C.~Drummond-Cole]{Gabriel C. Drummond-Cole}
\thanks{The first author was supported by IBS-R003-D1.}
\address{Center for Geometry and Physics, Institute for Basic Science (IBS), Pohang 37673, Republic of Korea}
\email{gabriel.c.drummond.cole@gmail.com}
\urladdr{https://drummondcole.com/gabriel/academic/}
\author[P.~Hackney]{Philip Hackney}
\thanks{The second author acknowledges the support of Australian Research Council Discovery Project grant DP160101519.}
\address{Department of Mathematics\\University of Louisiana at Lafayette\\ Lafayette, LA, United States of America}
\email{philip@phck.net} 
\urladdr{http://phck.net}
\date{June 28, 2019; revised August 10, 2021}

\begin{document}
\begin{abstract}
The functor between operadic algebras given by restriction along an operad map generally has a left adjoint.
We give a necessary and sufficient condition for the restriction functor to admit a right adjoint. 
The condition is a factorization axiom which roughly says that operations in the codomain operad can be written essentially uniquely as operations in arity one followed by operations in the domain operad.
\end{abstract}
\maketitle

\section{Introduction}
A map between (potentially colored) operads yields an associated restriction functor between their respective categories of algebras. The restriction functor is right adjoint to a functor which is a free extension along the operad map. 
Usually the restriction functor is not a \emph{left} adjoint. 
However, in some interesting examples, it is. 
Here we provide a necessary and sufficient condition for the existence of a right adjoint, and give an explicit construction of the right adjoint in the case that the criterion is satisfied.

The criterion is the existence and uniqueness of a certain kind of factorization. 
The main theorems below are stated in considerably more generality, but for the purposes of the introduction, we restrict to the monochrome version.
For the time being, assume that the ground category is a standard one, for instance abelian groups, topological spaces, spectra, or sets.
We use the evocative notation $\cp$ for the usual composition product of collections (see Definition~\ref{definition sub prod} and Notation~\ref{notation: cp sub Q}).
\begin{definition}
\label{definition: categorical extension, monochrome}
Let $\phi:\mathcal{P}\to {\calQ}$ be a map of (monochrome) operads.
We say $\phi$ is a \emph{monoidal extension} if the composition
\[
\mathcal{P}\cp_{\mathcal{P}(1)}{\calQ}(1)\to{\calQ}\cp_{{\calQ}}{\calQ}\cong {\calQ}
\]
is an isomorphism. 
\end{definition}
\begin{theorem}
\label{theorem: main theorem, monochrome}
Let $\phi:\mathcal{P}\to{\calQ}$ be a map of (monochrome) operads. 
The restriction functor $\phi^*$ from ${\calQ}$-algebras to $\mathcal{P}$-algebras is a left adjoint if and only if $\phi$ is a monoidal extension.
\end{theorem}
This result is, in fact, constructive. See Theorem~\ref{theorem: sufficient} and Section~\ref{section: examples}.

An illustrative special case (see Example~\ref{example: arity one}) is when $\phi$ is a homomorphism $R\to S$ of rings. 
It is well known that the restriction from $S$-modules to $R$-modules admits both a left adjoint (extension of scalars) and a right adjoint (coextension of scalars).
A ring homomorphism is automatically a monoidal extension, as the map $R\otimes_R S \to S\otimes_S S \cong S$ is an isomorphism of abelian groups.

This paper is organized as follows. 
In Section~\ref{section: notation} we fix notation for and describe the structure of colored collections and colored operads. 
Once these necessities are out of the way, in Section~\ref{section: theorem statements} we state the main theorem, including the construction of the right adjoint.  
Section~\ref{section: examples} contains several examples, including 
the inclusion of the Gerstenhaber operad into the Batalin--Vilkovisky operad, 
the inclusion of the colored operad governing operads into the colored operad governing cyclic operads, and 
the inclusion of the associative operad into the colored operad governing those operads which are concentrated in positive arity. 
In Sections~\ref{section: necessary} and~\ref{section: sufficient}, we prove necessity and sufficiency of our condition.

This paper is weakly inspired by an example in our paper~\cite{DrummondColeHackney:CFERIQMSAAIIC} which gives a sufficient condition for the existence of a right-induced model structure in the special case where the right adjoint is also a left adjoint. There, we gave the example of the forgetful functor from cyclic operads to operads (see Section~\ref{section: bijection on colors examples} below). The existence of a right adjoint to this forgetful functor, originally due to Templeton~\cite{Templeton:SGO}, surprised us as well as a number of experts with whom we discussed it. This forgetful functor is in fact a restriction functor along a map of colored operads, and we became curious about what features of the governing operad map enabled the existence of the right adjoint.

Ward~\cite[Proposition 7.9]{Ward:6OFGO}, gave a sufficient condition for the existence of this kind of right adjoint (assuming some restrictions on the ground category). 
Ward's motivations are different; his conditions are significantly more restrictive but they ensure not only the existence of the adjoint but its coincidence with a different functor which exists independently.  
Our condition here is both necessary and sufficient and applies in more general ground categories. 
See Theorems~\ref{theorem: necessary} and~\ref{theorem: sufficient} for the full details.

The analogous question for Lawvere theories was studied by Wraith in \cite[\S9]{Wraith:AT}, but the characterization there is quite different to ours.
It would be interesting to compare them in the area of overlap, namely when considering maps of monochrome operads in sets.

We are not aware of other references in the literature that study this question.

\subsection*{Conventions} 
We will use the notation $(\mathcal{E},\otimes, \mathbf{1})$ for a bicomplete symmetric monoidal category, which we abbreviate as $\mathcal{E}$.
All such $\mathcal{E}$ will be assumed closed and we write $\twohoms{-}{-}$ for the internal hom.
We use $\mathbb{N}$ to denote the set of nonnegative integers.
For $n\in \mathbb{N}$, we write $[n]$ for the set $\{1,2,\dots, n\}$ and $\Sigma_n$ for the symmetric group on $n$ letters, $\Sigma_n = \operatorname{Aut}([n])$.

\subsection*{Acknowledgments}
The authors would like to thank Michael Batanin, Rune Haugseng, Robin Koytcheff, Damien Lejay, Marcy Robertson, Ben Ward, and Donald Yau for useful discussions. 
We would also like to thank Gavin Wraith for pointing out his earlier related work.
We are grateful to the referee, whose comments helped us improve the clarity and precision of this paper.

\section{Operads and collections}
\label{section: notation}
In this section we establish conventions for colored operads. 
The parts for fixed colors are fairly standard.
Some of the machinery related to the change of colors is less standard, although not much harder.

\subsection{Collections}
\begin{definition}\label{def: collections}
Suppose that $\colorsa$ is a set of \emph{colors}, and let $\listsopa$ denote the groupoid whose objects are lists $\ua \coloneqq a_1, \dots, a_n$ (where $n$ varies and $a_i \in \colorsa$) and whose morphisms are
\[
	\ua = a_1, \dots, a_n \xrightarrow{\sigma} a_{\sigma(1)},\dots, a_{\sigma(n)} \eqqcolon \ua \sigma
\]
where $\sigma\in \Sigma_n$ is a permutation.
\begin{itemize}
\item If $\colorsb$ is another color set, then an \emph{$(\colorsa,\colorsb)$-collection} $X$ is an object in the functor category $\mathcal{E}^{\listsopa \times \colorsb}$.
\item
The maps of $(\colorsa,\colorsb)$-collections are natural transformations between the functors.
\item
If $\colorsa=\emptyset$, we call an $(\emptyset,\colorsb)$-collection a \emph{$\colorsb$-object}, and write $X_b$ for $X(\,\,;b)$. 
\item We call an object $(a_1,\ldots, a_n; b)$ in $\listsopa\times\colorsb$ an \emph{$(\colorsa,\colorsb)$-profile} or just a profile if $\colorsa$ and $\colorsb$ are clear from context.
We will alternately write such a profile as $(\ua;b)$ or $\vprof{\ua}{b}$, so if $X$ is an $(\colorsa,\colorsb)$-collection we will write $X\vprof{\ua}{b}$ or $X(\ua; b)$ for the value of $X$ at the indicated profile.
\end{itemize}
\end{definition}

Concretely, an $(\colorsa,\colorsb)$-collection $X$ consists of: 
\begin{enumerate}
\item for each (possibly empty) list $(a_1,\ldots, a_n)$ of colors in the color set $\colorsa$ and each color $b$ in $\colorsb$, an object $X(a_1,\ldots, a_n;b)$ of $\mathcal{E}$, and
\item for each element $\sigma$ of the symmetric group $\Sigma_n$, color $b$ in $\colorsb$, and tuple $(a_1,\ldots, a_n)$ in $\colorsa$, a morphism $\sigma^*$ in $\mathcal{E}$ of the form
\[
X(a_1,\ldots, a_n;b)\to X(a_{\sigma(1)},\ldots,a_{\sigma(n)};b)
\]
such that $\id^* = \id$ and $\sigma^*\tau^*=(\tau\sigma)^*$ for all $\sigma$ and $\tau$ in $\Sigma_n$.
\end{enumerate}
A map $X\to Y$ consists of, for each color $b$ in $\colorsb$ and each list $(a_1,\ldots, a_n)$ of colors in $\colorsa$, a morphism in $\mathcal{E}$ from $X(a_1,\ldots, a_n;b)$ to $Y(a_1,\ldots, a_n;b)$ such that the evident $\Sigma_n$ equivariance conditions are satisfied.

\begin{notation}\label{notation: collection associated to a function}
Let $f:\colorsa\to\colorsb$ be a map between color sets. 
We can build two collections out of $f$ as follows.
We build an $(\colorsa,\colorsb)$-collection also called $f$:
\[
f(a_1,\ldots,a_n;b)=
\begin{cases}
\mathbf{1}
&
\text{if }n=1\text{ and }f(a_1)=b,
\\
\initialE
&
\text{otherwise,}
\end{cases}
\]
where $\initialE$ is the initial object of $\mathcal{E}$.
We also have a $(\colorsb,\colorsa)$-collection $\bar{f}$
\[
\bar{f}(b_1,\ldots,b_n;a)=
\begin{cases}
\mathbf{1}
&
\text{if }n=1\text{ and }f(a)=b_1,
\\
\initialE
&
\text{otherwise.}
\end{cases}
\]
When $f$ is invertible, the collections $f^{-1}$ and $\bar{f}$ are canonically isomorphic.

By convention we use the notation $\mathbf{1}_{\colorsa}$ for the $(\colorsa,\colorsa)$-collection $\id_\colorsa$.
\end{notation}

More generally, if $p$ is any span of sets from $\colorsa$ to $\colorsb$, then there is an associated $(\colorsa,\colorsb)$-collection concentrated in arity one. 
Writing $p : \mathsf{U} \to \colorsa \times \colorsb$, this collection is given in profile $(a;b)$ by $\coprod_{p^{-1}(a,b)} \mathbf{1} = p^{-1}(a,b) \cdot \mathbf{1}$ (where the $\cdot$ denotes copower).
Then $f$ comes from the span $(\id, f) : \colorsa \to \colorsa \times \colorsb$ while $\bar{f}$ comes from the span $(f,\id) : \colorsa \to \colorsb \times \colorsa$.

Above we used the groupoid $\listsopa$ of lists of colors.
We also have the groupoid $\listsa$ whose morphisms are $\sigma : \ua \to \sigma \ua \coloneqq a_{\sigma^{-1}(1)}, \dots, a_{\sigma^{-1}(n)}$.
Notice there are functors from $\listsa$ to the discrete category $\mathbb{N} = \{0,1,2,\dots \}$ that take a list $\ua$ to its length.
Both categories $\listsopa$ and $\listsa$ are strict monoidal categories (in fact, free strict symmetric monoidal categories).
The following definition is essentially adapted from \cite[\S2]{Kelly:OOJPM}, and appears in the colored case when $\colorsa=\colorsb$ in \cite[\S3.1]{Yau:IOMCGE}.
\begin{definition}[Day powers]
\label{definition day powers}
Let $Y$ be an $(\colorsa, \colorsb)$-collection.
\begin{itemize}
\item For each $\ub \in \listsb$ of length $m$ and each $\ua \in \listsopa$, there is a functor 
\[
	F : \left\lgroup \prod_{j=1}^m \listsopa \right\rgroup^{\oprm} \times \left\lgroup \prod_{j=1}^m \listsopa \right\rgroup \to \mathcal{E}
\]
having value
\[
	F(\{\uu_j\}, \{\uv_j\}) = \listsopa	(\uu_1\uu_2\dots\uu_m; \ua) \cdot \bigotimes_{j=1}^m Y\vprof{\uv_j}{b_j}.
\]
Here the $\cdot$ denotes copower and $\uu_1\uu_2\dots\uu_m$ is the concatenation of the lists.
\item If $\ua$ is a list of elements in $\colorsa$ and $\ub$ is a list of elements in $\colorsb$, we define an object 
\[
	Y^{\ub}(\ua) \coloneqq \int^{\{\uw_j\} \in \prod\limits_{j=1}^m \listsopa} \listsopa (\uw_1\uw_2\dots\uw_m; \ua) \cdot \bigotimes_{j=1}^m Y\vprof{\uw_j}{b_j}
\]
of $\mathcal{E}$ as the coend of the bifunctor $F$.\footnote{This superscript notation matches with the $S^m$ and $(S^m)k$ appearing in \cite{Kelly:OOJPM}.}
There is an evident naturality in the $\ua$ variable given by postcomposition; there is also a naturality in the $\ub$ variable which makes this a functor $\listsopa \times \listsb \to \mathcal{E}$.
\end{itemize}
\end{definition}

\begin{definition}[Kelly]
\label{definition sub prod}
Suppose that $X$ is a $(\colorsb, \colorsc)$-collection and $Y$ is an $(\colorsa, \colorsb)$-collection.
The \emph{composition product} of $X$ and $Y$ is defined to be the $(\colorsa, \colorsc)$-collection 
\[
	(X \cp Y) \vprof{\ua}{c} = \int^{\ub \in \listsb} X\vprof{\ub}{c} \otimes Y^{\ub}(\ua).
\]
\end{definition}

\begin{remark}
The functor $(-) \cp Y$ goes from $(\colorsb, \colorsc)$-collections to $(\colorsa, \colorsc)$-collections and has a right adjoint, temporarily denoted $\{Y, - \}$.
If $Z$ is an $(\colorsa, \colorsc)$-collection, then the $(\colorsb, \colorsc)$-collection $\{ Y, Z \}$ is given by the end
\[
	\{ Y, Z \} \vprof{\ub}{c} = \int_{\ua \in \listsopa} \twohoms{Y^{\ub}(\ua)}{Z\vprof{\ua}{c}}
\]
where square brackets denote the internal hom of $\mathcal{E}$.
We will frequently need that $(-) \cp Y$ is a left adjoint in what follows, but we will never explicitly use the formula for $\{ Y, Z \}$.
\end{remark}

In general $X \cp (-)$ is not a left adjoint functor (see \cite[p.7]{Kelly:OOJPM}), but is when $X$ is concentrated in arity one (see Lemma~\ref{one right adjoint}).

\begin{lemma}
\label{lemma: Day power when Y is arity one}
Suppose $Y$ is concentrated in arity one. Then $Y^{\ub}(\ua)$ is naturally isomorphic to
\[
\coprod_{\sigma \in \Sigma_m} \bigotimes_{j=1}^m Y(a_{\sigma^{-1}(j)}; b_j).
\]
\end{lemma}
\begin{proof}
By definition,
\[
	Y^{\ub}(\ua)\cong \int^{\{w_j\} \in \prod\limits_{j=1}^m \colorsa} \listsopa (w_1w_2\dots w_m; \ua) \cdot \bigotimes_{j=1}^m Y\vprof{w_j}{b_j}
\]
and the indexing category is discrete so this is isomorphic to
\begin{align*}
			\coprod_{\uw \in \colorsa^{\times m}} \coprod_{\listsa (\ua; \uw)} \bigotimes_{j=1}^m Y\vprof{w_j}{b_j}.
\end{align*}
We can identify $\listsa (\ua; \uw)$ with the set of $\sigma\in \Sigma_m$ so that $w_j = a_{\sigma^{-1}(j)}$.
The set of pairs $(\uw, \sigma) \in \colorsa^{\times m} \times \Sigma_m$ satisfying $w_j = a_{\sigma^{-1}(j)}$ is just in bijection with $\Sigma_m$.
\end{proof}
On the other hand, if $X$ is concentrated in arity one and we are investigating $X\cp (-)$, then we are only concerned with the case of $m=1$.
In general we have
\[
	Y^b(\ua) = \int^{\uw \in \listsopa} \listsopa (\uw; \ua) \cdot  Y\vprof{\uw}{b} \cong Y\vprof{\ua}{b}
\]
and we see that 
\[
	(X \cp Y)\vprof{\ua}{c} = \int^{b \in \colorsb} X\vprof{b}{c} \otimes Y^{b}(\ua) = \coprod_{b\in \colorsb} X\vprof{b}{c} \otimes Y\vprof{\ua}{b}.
\]
\begin{notation}
\label{notation for angle brackets}
Suppose that $X$ is a $(\colorsb,\colorsc)$-collection concentrated in arity one and let $Z$ be an $(\colorsa,\colorsc)$-collection.
Define $\langle X, Z \rangle$ to be the $(\colorsa,\colorsb)$-collection given by
\[
	\langle X, Z \rangle \vprof{\ua}{b} = \prod_{c\in \colorsc} \twohoms{X\vprof{b}{c}}{Z\vprof{\ua}{c}}.
\]
\end{notation}
\begin{lemma}
\label{one right adjoint}
Suppose that $X$ is a $(\colorsb,\colorsc)$-collection concentrated in arity one and let $\colorsa$ be a set of colors.
The functor $\langle X, - \rangle$ from $(\colorsa,\colorsc)$-collections to $(\colorsa,\colorsb)$-collections is right adjoint to $X\cp(-)$.
\end{lemma}

When $\colorsa = \colorsb = \colorsc$, Definition~\ref{definition sub prod} agrees with the colored circle product $\circ$ from 
\cite[2.2.3]{HackneyRobertsonYau:RLPCO}, the colored symmetric circle product $\circ^{\mathsf{S}}$ from \cite[\S 3.1]{Yau:IOMCGE}, and the $\Box$-product from \cite[7.2]{BergerMoerdijk:RCORHA}.
Of course the category of $(\colorsa,\colorsa)$-collections equipped with this product is a monoidal category (see \cite[\S 3]{Kelly:OOJPM}, \cite[Proposition 2.1.8]{Yau:IOMCGE}, or the sources above).
Essentially the same proof of that fact shows that the collection of $(\colorsa,\colorsb)$-collections as $\colorsa$ and $\colorsb$ varies forms a bicategory. 
Indeed, there are natural associator and unitor isomorphisms for $\cp$ and $\mathbf{1}_\colorsa$ and we have:
\begin{definition}
The bicategory of colored collections $\CoCo$ has data defined as follows.
\begin{itemize}
\item The $0$-cells are sets of colors.
\item The $1$-cells from color $\colorsa$ to color $\colorsb$ are $(\colorsa, \colorsb)$-collections.
\item The $2$-cells from an $(\colorsa,\colorsb)$-collection $X$ to an $(\colorsa,\colorsb)$-collection $Y$ are maps of collections from $X$ to $Y$.
\item The horizontal composition is $\cp$.
\item The identity morphism for color $\colorsa$ is $\mathbf{1}_{\colorsa}$.
\end{itemize}
The associators and unitors will not be described explicitly.
\end{definition}

Recall from Notation~\ref{notation: collection associated to a function} that every function determines a collection in two different ways.
When considered as 1-cells of the bicategory $\CoCo$ these two collections are adjoint.

\begin{example}
\label{example: adjunctions of colors}
Let $f:\colorsa\to\colorsb$ be a function between color sets. Then the compositions of the collections $f$ and $\bar{f}$ are as follows:
\begin{align*}
\left(f \cp \bar{f}\right)(b_1,\ldots, b_n; b)
&\cong
\begin{cases}
\coprod_{a\in f^{-1}(b)} \mathbf{1} 
&\text{if }n=1\text{ and }b_1=b,
\\
\initialE & \text{otherwise.}
\end{cases}
\\
\left(\bar{f} \cp f\right)(a_1,\ldots, a_n; a)
&\cong
\begin{cases}
\mathbf{1} 
&\text{if }n=1\text{ and }f(a_1)=f(a),
\\
\initialE & \text{otherwise.}
\end{cases}
\end{align*}
\end{example}

Notice that these collections may also be obtained by first composing the span $(\id,f)$ with its reverse $(f,\id)$ and then taking the corresponding collection.

\begin{definition}
\label{definition: adjunction of colors}
The \emph{canonical counit} $\epsilon_f$ of $f$ is the map of collections $f \cp \bar{f}\to \mathbf{1}_\colorsb$ induced in each profile by the fold map $\coprod \mathbf{1}\to \mathbf{1}$. 
The \emph{canonical unit} $\eta_f$ of $f$ is the inclusion of collections $\mathbf{1}_{\colorsa}\to \bar{f} \cp {f}$. 
\end{definition}
\begin{lemma} 
\label{lemma color map adjunctions}
If $f$ is a map of colors $\colorsa\to\colorsb$, then the one-cell $f$ in the bicategory $\CoCo$ is left adjoint to the one-cell $\bar{f}$.
As a consequence, we have the following induced adjunctions:
\begin{itemize}
\item The functor $f\cp -$ from $(\colorsc,\colorsa)$-collections to $(\colorsc,\colorsb)$-collections is left adjoint to the functor $\bar{f}\cp-$ in the opposite direction.
\item The functor $-\cp \bar{f}$ from $(\colorsa,\colorsc)$-collections to $(\colorsb,\colorsc)$-collections is left adjoint to the functor $-\cp f$ in the opposite direction.
\end{itemize}
These functors are equivalences of categories if and only if $f$ is invertible.
\end{lemma}
\begin{proof}
The canonical unit and counit are compatible: the compositions 
\[f\xrightarrow{f \cp \eta_f} f \cp \bar{f} \cp f\xrightarrow{\epsilon_f\cp f}{f}\]
and
\[
\bar{f}\xrightarrow{\eta_f\cp \bar{f}} \bar{f} \cp f \cp \bar{f}\xrightarrow{\bar{f} \cp\, \epsilon_f}\bar{f}
\]
are identity maps. 
The canonical unit and counit are isomorphisms if and only if $f$ is invertible.
\end{proof}

% The observant reader will note that $\bar{f}\cp-$ is isomorphic to $\langle f, - \rangle$ by Lemma~\ref{one right adjoint}.

\subsection{Operads}
We now give a definition of (colored) operad which is convenient for our purposes, as well as several descriptions of maps of such.
\begin{definition}
\label{definition: operads}
A colored operad is a monad (as in~\cite[5.4.1]{Benabou:IB}) in the bicategory of colored collections 
(see also~\cite{Street:FTM}, which we will use below for morphisms of monads). 
In other words, it is a choice of color set $\colorsa$ and a monoid in the monoidal category consisting of $(\colorsa,\colorsa)$-collections along with $\cp$ and $\mathbf{1}_{\colorsa}$. 
More explicitly, the data is given by an $\colorsa$-colored collection $\mathcal{P}$ and maps of collections $\mu_\mathcal{P} : \mathcal{P}\cp\mathcal{P}\to\mathcal{P}$ (called \emph{composition}) and $\eta_\mathcal{P} : \mathbf{1}_\colorsa\to \mathcal{P}$ (called the \emph{unit}) which satisfy associativity and unit constraints.
This definition essentially appeared in the monochrome case in \cite{Kelly:OOJPM,Smirnov:HTC}, while the colored case appeared in \cite[\S 7.3]{BergerMoerdijk:RCORHA}.

A map of colored operads $(\colorsa,\mathcal{P})\to(\colorsb,{\calQ})$ is a pair $(f,\phi)$ where $f$ is a function from $\colorsa$ to $\colorsb$ and $\phi$ is a $2$-cell in $\CoCo$ from $f \cp \mathcal{P}$ to ${\calQ}\cp f$ so that the following two diagrams commute (up to suppressed associators and unitors).
\[
\begin{tikzcd}
& f \cp \mathcal{P}\ar[dd,"\phi"]
&f \cp \mathcal{P}\cp\mathcal{P}\ar[dd,swap,"f \cp \mu_\mathcal{P}"]
\ar[r,"\phi\cp\mathcal{P}"]
&{\calQ}\cp f \cp \mathcal{P}\ar[r,"{\calQ}\cp\phi"]
&{\calQ}\cp{\calQ}\cp f\ar[dd,"\mu_{\calQ}\cp f"]
\\
f\ar[ur,"f \cp \eta_\mathcal{P}"]\ar[dr,swap,"\eta_{\calQ}\cp f"]&&
\\
& {\calQ}\cp f&f \cp \mathcal{P}\ar[rr,swap,"\phi"]
&&{\calQ}\cp f
\end{tikzcd}
\]
In other words, it is a pair $(f,\phi)$ such that $(f,\phi)$ is a \emph{monad opfunctor}~\cite[\S4]{Street:FTM} or \emph{colax map of monads}~\cite[\S6.1]{Leinster:HOHC}.
\end{definition}
\begin{remark}[Warning]
Not all monad opfunctors are maps of operads, because the one-cell $f$ must be of a particular form, i.e., must come from a map of color sets. 
As pointed out to us by Rune Haugseng, there is a double categorical framework enhancing $\CoCo$ which takes a little more setup in which the presentation is more uniform.
\end{remark}
Unpacking the definition further using the adjunctions of Lemma~\ref{lemma color map adjunctions}, the data of $\phi$ consists of the following. 
For each profile $(a_1,\ldots,a_k;a)$ of colors in $\colorsa$, we are given a map in $\mathcal{E}$ from $\mathcal{P}(a_1,\ldots,a_k;a)$ to ${\calQ}(f(a_1),\ldots, f(a_k);f(a))$. 
This map is called the \emph{component of $\phi$} at $(a_1,\ldots, a_k;a)$. 
Commutativity of the triangle says that the component of $\phi$ at $(a;a)$ should take the unit at $a$ to the unit at $f(a)$.
Commutativity of the pentagon says that this collection of maps should intertwine the composition of $\mathcal{P}$ and the composition of ${\calQ}$.

\begin{example}
If $\colorsa$ is any set of colors, then the $\colorsa$-colored collection $\mathbf{1}_{\colorsa}$ is an operad with $\eta$ the identity morphism and $\mu$ an instance of the unitor isomorphism.
\end{example}

\begin{remark}[Alternative presentations of maps of colored operads]\label{remark: alt pres col op map}
Let's give two equivalent definitions of a map from \((\colorsa,\mathcal{P})\) to  \((\colorsb,{\calQ})\).
In both cases we will have pair consisting of a function $f: \colorsa\to \colorsb$ and also a 2-cell in $\CoCo$.
That these are equivalent to the original definition is an exercise using the fact that the 1-cell $f$ is left adjoint to the $1$-cell $\bar{f}$ in the bicategory $\CoCo$ by Lemma~\ref{lemma color map adjunctions}.
\begin{enumerate}
\item The 2-cell $\chi$ goes from $\mathcal{P}$ to $\bar{f} \cp {\calQ} \cp f$ and the following two diagrams commute.
\begin{gather*} 
\begin{tikzcd}[column sep=large,ampersand replacement=\&]
\mathbf{1}_{\colorsa} \rar{\eta_{\mathcal{P}}} \dar{\eta_f} \& \mathcal{P} \dar{\chi} \\
\bar{f} \cp {f} \rar{\bar{f} \cp \eta_{{\calQ}} \cp f} \& \bar{f} \cp {\calQ} \cp {f} 
\end{tikzcd} 
\\
\begin{tikzcd}[ampersand replacement=\&]
\mathcal{P} \cp \mathcal{P} 
\dar{\mu_{\mathcal{P}}} \rar{\chi\cp \chi}
\& \bar{f} \cp {\calQ} \cp f \cp \bar{f} \cp {\calQ} \cp f 
\rar{\bar{f} \cp {\calQ} \cp \epsilon_f \cp {\calQ} \cp f }
\&[+2em] \bar{f} \cp {\calQ} \cp {\calQ} \cp f \dar{\bar{f} \cp \mu_{{\calQ}} \cp f}
\\
\mathcal{P} \arrow[rr,"\chi"] \&  \& \bar{f} \cp {\calQ} \cp f
\end{tikzcd}
\end{gather*}
\item The 2-cell $\psi$ goes from $\mathcal{P} \cp \bar{f}$ to $\bar{f} \cp {\calQ}$ and the following two diagrams commute.
\[ \begin{tikzcd}
& \mathcal{P} \cp \bar{f} \ar[dd,"\psi"]
\\
\bar{f}\ar[ur,"\eta_\mathcal{P} \cp \bar{f}"]\ar[dr,swap,"\bar{f} \cp \eta_{\calQ}"]
\\
& \bar{f} \cp {\calQ}
\end{tikzcd} 
\qquad \qquad
\begin{tikzcd}
\mathcal{P} \cp \mathcal{P} \cp \bar{f} \rar{\mathcal{P} \cp \psi} \dar{\mu_{\mathcal{P}} \cp \bar{f}}
 & \mathcal{P} \cp \bar{f} \cp {\calQ} \rar{\psi \cp {\calQ}} & \bar{f} \cp {\calQ} \cp {\calQ} \dar{\bar{f} \cp \mu_{{\calQ}}}
\\
\mathcal{P} \cp \bar{f} \arrow[rr,"\psi"] && \bar{f} \cp {\calQ}
\end{tikzcd}
\]
\end{enumerate}
The first of these redefinitions is just saying that a colored operad map $(\colorsa,\mathcal{P})\to(\colorsb,{\calQ})$ is the same thing as a colored operad map $(\colorsa,\mathcal{P})\to(\colorsa,\bar{f}\cp{\calQ}\cp f)$ which is the identity on color sets.
\end{remark}

Note that there is little distinction between the three versions of colored operad map when $f$ is the identity function on $\colorsa$.

\subsection{Modules and algebras}
Suppose that $(\colorsa,\mathcal{P})$ and $(\colorsb,{\calQ})$ are two colored operads.
As with monads in any bicategory, there are good notions of left $\mathcal{P}$-modules, right ${\calQ}$-modules, and $\mathcal{P}\text{-}{\calQ}$ bimodules.
\begin{definition}
Let $(\colorsa,\mathcal{P})$ and $(\colorsb,{\calQ})$ be colored operads.
\begin{itemize}
\item 
A left $\mathcal{P}$-module consists of a $(\colorsc, \colorsa)$-collection $X$ for some color set $\colorsc$ together with a map of $(\colorsc,\colorsa)$-collections $\lambda: \mathcal{P} \cp X \to X$ so that the diagrams
\[ 
\begin{tikzcd}
\mathcal{P} \cp \mathcal{P} \cp X \rar{\mu_{\mathcal{P}}\cp X} \dar{\mathcal{P} \cp \lambda} & \mathcal{P} \cp X\dar{\lambda}  \\
\mathcal{P} \cp X \rar{\lambda} & X 
\end{tikzcd} 
\qquad
\begin{tikzcd}
\mathbf{1}_{\colorsa} \cp X \arrow[dr,"\cong"] \rar{\eta_{\mathcal{P}}\cp X} & \mathcal{P} \cp X \dar{\lambda} \\
& X
\end{tikzcd}
\]
commute.
We also call such a $\lambda$ a left $\mathcal{P}$-action.
\item
Likewise, a right ${\calQ}$-module consists of a $(\colorsb,\colorsc)$-collection $X$ for some color set $\colorsc$ 
together with a map of $(\colorsb,\colorsc)$-collections $\rho:  X \cp {\calQ} \to X$ so that the diagrams
\[ 
\begin{tikzcd}
X \cp {\calQ} \cp {\calQ} \rar{X\cp \mu_{{\calQ}}} \dar{\rho \cp {\calQ} } & X \cp {\calQ} \dar{\rho}  \\
X \cp {\calQ} \rar{\rho} & X 
\end{tikzcd} 
\qquad
\begin{tikzcd}
X \cp \mathbf{1}_{\colorsb} \arrow[dr,"\cong"] \rar{X \cp \eta_{\calQ}} & X \cp {\calQ} \dar{\rho} \\
& X
\end{tikzcd}
\]
commute.
We also call such a $\rho$ a right ${\calQ}$-action.
\item 
A $\mathcal{P}\text{-}{\calQ}$ bimodule is a $(\colorsb,\colorsa)$-collection which is both a left $\mathcal{P}$-module and a right ${\calQ}$-module so that the square
\[ \begin{tikzcd}
\mathcal{P} \cp X \cp {\calQ} \rar{\mathcal{P} \cp \rho} \dar{\lambda \cp {\calQ}} & \mathcal{P} \cp X \dar{\lambda} \\
X \cp {\calQ} \rar{\rho} & X 
\end{tikzcd} \]
commutes.
Write $\bimod{\mathcal{P}}{{\calQ}}$ for the category of bimodules.
\item 
An \emph{algebra over $\mathcal{P}$} is an $(\emptyset,\colorsa)$-collection $\mathcal{A}$ (that is, an $\colorsa$-object), equipped with the structure of a left $\mathcal{P}$-module.
We write $\algebras{\mathcal{P}}$ for the category of $\mathcal{P}$-algebras.
\end{itemize}
\end{definition}
\begin{remark}
\leavevmode
\begin{enumerate}
\item 
Concretely, the data of an algebra $\mathcal{A}$ is given by an $\colorsa$-indexed family of $\mathcal{E}$-objects $\mathcal{A}_a \coloneqq \mathcal{A}(\,\,;a)$ along with maps
\[
\mathcal{P}(a_1,\ldots, a_n;a)\otimes \mathcal{A}_{a_1}\otimes\cdots\otimes \mathcal{A}_{a_n} \to \mathcal{A}_a
\]
which satisfy associativity, unitality, and equivariance constraints.
\item 
Note that $\algebras{\mathcal{P}}$ is nothing but $\bimod{\mathcal{P}}{{\calQ}}$ for $\mathcal{Q}$ the initial colored operad, that is, the unique operad with an empty color set.
Many of our results are stated for $\algebras{\mathcal{P}}$, but hold more generally for $\bimod{\mathcal{P}}{\mathcal{R}}$ where $\mathcal{R}$ is some operad which stays fixed throughout the discussion.
\item
Every $(\colorsc, \colorsa)$-collection $X$ is automatically a right $\mathbf{1}_{\colorsc}$-module, in a way that is compatible with any left $\mathcal{P}$-module structure on $X$.
\end{enumerate}
\end{remark}
\begin{notation}\label{notation: cp sub Q}
Suppose given three colored operads $\mathcal{P}$, ${\calQ}$, and $\mathcal{R}$.
Then there is an induced functor
\[
\cp_{{\calQ}} : \bimod{\mathcal{P}}{{\calQ}} \times \bimod{{\calQ}}{\mathcal{R}} \to \bimod{\mathcal{P}}{\mathcal{R}}
\]
given at the level of collections as a reflexive coequalizer
\[ \begin{tikzcd}
X \cp {\calQ} \cp Y \rar[shift left, "X \cp \lambda_Y "] \rar[shift right, "\rho_X\cp Y" swap] & X \cp Y \rar & X \cp_{{\calQ}} Y.
\end{tikzcd} \]
\end{notation}
The special case when ${\calQ} = \mathbf{1}_{\colorsb}$ gives $X\cp_{\mathbf{1}_{\colorsb}}Y \cong  X\cp Y$.

\begin{proposition}
\label{proposition associativity}
If ${\calQ}$ and $\mathcal{R}$ are two colored operads, then
\[
	(X \cp_{{\calQ}} Y)\cp_{\mathcal{R}} Z \cong X \cp_{{\calQ}} (Y\cp_{\mathcal{R}} Z).
\]
\end{proposition}
\begin{proof}
This essentially follows from the fact that $W\cp(-)$ preserves \emph{reflexive coequalizers} while $(-)\cp W$ preserves all colimits (in particular, reflexive coequalizers).
\end{proof}

\subsection{Restriction and extension of algebras}
A map of operads induces a well-known adjunction between their categories of algebras. In this section we review this procedure, describing it explicitly and categorically.
\begin{lemma}
\label{lemma: left module structure}
Let $(f,\phi)$ be a map of operads from $\mathcal{P}$ to ${\calQ}$. There is an induced left $\mathcal{P}$-action $\lambda_\phi$ on $\bar{f} \cp {\calQ}$.
\end{lemma}
\begin{proof}
We utilize the second adjoint definition from Remark~\ref{remark: alt pres col op map} and assume we have a 2-cell $\psi : \mathcal{P} \cp \bar{f} \to \bar{f} \cp {\calQ}$
so that the two diagrams
\[ \begin{tikzcd}[row sep=tiny]
& \mathcal{P} \cp \bar{f} \ar[dd,"\psi"]
\\
\bar{f}\ar[ur,"\eta_\mathcal{P} \cp \bar{f}"]\ar[dr,swap,"\bar{f} \cp \eta_{\calQ}"]
\\
& \bar{f} \cp {\calQ}
\end{tikzcd} 
\qquad \qquad
\begin{tikzcd}
\mathcal{P} \cp \mathcal{P} \cp \bar{f} \rar{\mathcal{P} \cp \psi} \dar{\mu_{\mathcal{P}} \cp \bar{f}}
 & \mathcal{P} \cp \bar{f} \cp {\calQ} \rar{\psi \cp {\calQ}} & \bar{f} \cp {\calQ} \cp {\calQ} \dar{\bar{f} \cp \mu_{{\calQ}}}
\\
\mathcal{P} \cp \bar{f} \arrow[rr,"\psi"] && \bar{f} \cp {\calQ}
\end{tikzcd}
\]
commute.
Define $\lambda$ to be the composite
\[
	\mathcal{P} \cp \bar{f} \cp {\calQ} \xrightarrow{\psi \cp {\calQ}} \bar{f} \cp {\calQ} \cp {\calQ} \xrightarrow{\bar{f} \cp \mu_{{\calQ}}} \bar{f} \cp {\calQ}.
\]
We see that $\lambda (\mathcal{P} \cp \lambda) = \lambda (\mu_{\mathcal{P}} \cp \bar{f} \cp {\calQ})$ holds because of commutativity of the following diagram.

\[ \begin{tikzcd}[column sep=large]
\mathcal{P} \cp \mathcal{P} \cp \bar{f} \cp {\calQ} \arrow[dd,"\mu_{\mathcal{P}} \cp \bar{f} \cp {\calQ}"] \rar{\mathcal{P} \cp \psi \cp {\calQ}} & 
\mathcal{P} \cp \bar{f} \cp {\calQ} \cp {\calQ} \rar{\mathcal{P} \cp \bar{f} \cp \mu_{{\calQ}}} \dar{\psi \cp {\calQ} \cp {\calQ} }& 
\mathcal{P} \cp \bar{f} \cp {\calQ} \dar{\psi \cp {\calQ}}
\\
& 
\bar{f} \cp {\calQ} \cp {\calQ} \cp {\calQ} \rar{\bar{f} \cp {\calQ} \cp \mu_{{\calQ}}} \dar{\bar{f} \cp \mu_{{\calQ}} \cp {\calQ}} & 
\bar{f} \cp {\calQ} \cp {\calQ} \dar{\bar{f} \cp \mu_{{\calQ}}}
\\
\mathcal{P} \cp \bar{f} \cp {\calQ} \rar{\psi \cp {\calQ}} & 
\bar{f} \cp {\calQ} \cp {\calQ} \rar{\bar{f} \cp \mu_{{\calQ}}} & 
\bar{f} \cp {\calQ}
\end{tikzcd} \]
Here the left-hand rectangle commutes by our assumption on $\psi$, while the two small squares commute by naturality.

Compatibility with the unit holds by the diagram
\[ \begin{tikzcd}
& \mathcal{P} \cp \bar{f} \cp {\calQ} \arrow[ddr, bend left, "\lambda"] \dar{\psi \cp {\calQ}} \\
& \bar{f} \cp {\calQ} \cp {\calQ} \arrow[dr,"\bar{f} \cp \mu_{{\calQ}}" description]\\
\bar{f} \cp {\calQ} \arrow[uur, "\eta_\mathcal{P} \cp \bar{f} \cp {\calQ}", bend left] \arrow[ur, "\bar{f} \cp {\calQ} \cp \eta_{\calQ}" description]  \arrow[rr, "\id"] & & 
\bar{f} \cp {\calQ}.
\end{tikzcd} \]
\par \vspace{-1.5\baselineskip} \qedhere
\end{proof}

\begin{remark}[Induced right module structure]
\label{remark: right module structure}
Similarly, given an operad map $\mathcal{P} \to {\calQ}$ there is an induced right $\mathcal{P}$-action $\rho_\phi$ on ${\calQ}\cp f$.
The proof is akin to the proof of Lemma~\ref{lemma: left module structure}, except one should use the original definition of operad map rather than its adjoint version.
\end{remark}

\begin{remark}[The usual adjunction between algebras]
\label{remark: usual adjunction between algebras}
Suppose that $(f,\phi): \mathcal{P} \to {\calQ}$ is a map of operads.
In light of Lemma~\ref{lemma: left module structure} there is a functor $\phi^*$ from ${\calQ}$-algebras to $\mathcal{P}$-algebras given on underlying collections by
\[
	\mathcal{B} \mapsto (\bar{f}\cp {\calQ})\cp_{{\calQ}} \mathcal{B} \cong \bar{f}\cp \mathcal{B},
\]
with action induced by the left $\mathcal{P}$-module structure on $\bar{f}\cp {\calQ}$.

Similarly, in light of Remark~\ref{remark: right module structure}, there is a functor $\phi_!$ from $\mathcal{P}$-algebras to ${\calQ}$-algebras given on underlying collections by
\[
	\mathcal{A} \mapsto ({\calQ}\cp f) \cp_{\mathcal{P}} \mathcal{A}.
\]
The outcome of this process retains the evident left ${\calQ}$-action.

Using the canonical unit and counit of $f$, it can be shown that $\phi_!$ is left adjoint to $\phi^*$.
\end{remark}
We are interested in the question of when $\phi^*$ admits an exceptional \emph{right} adjoint $\phi_*$ in addition to the left adjoint $\phi_!$.
We next explore one situation where this occurs, namely when $\mathcal{P}$ and ${\calQ}$ are actually categories and $\phi_*$ is right Kan extension.

\subsection{Categories as operads}

\begin{definition}
Let $\mathcal{P}$ be an $\colorsa$-colored operad. 
The \emph{underlying category} $\underlying{\mathcal{P}}$ is the $\mathcal{E}$-enriched category whose
\begin{itemize}
\item objects are the elements of the color set $\colorsa$,
\item morphism object between $c$ and $d$ is $\mathcal{P}(c;d)$, and
\item whose unit and composition are given by the unit and the restriction of the composition of $\mathcal{P}$.
\end{itemize}
In the same spirit, if $X$ is an $(\colorsa,\colorsb)$-collection, we will write $\underlying{X}$ for the corresponding collection that is concentrated in arity one.
That is, $\underlying{X}$ is the collection with
\[
\underlying{X}(a_1, \dots, a_n; b) = \begin{cases}
	X(a_1; b) & \text{ if $n=1$} \\
	\initialE & \text{ otherwise.}
\end{cases}
\]
\end{definition}
The underlying category assignment $\underlying{-}$ is a functor, and is right adjoint to the inclusion of categories into operads. 
We can also think of the underlying category $\underlying{\mathcal{P}}$ as an $\colorsa$-colored suboperad of $\mathcal{P}$ via the counit of the adjunction, which we will do without comment or change of notation.

Let $\mathcal{M}$ and $\mathcal{N}$ be categories enriched in $\mathcal{E}$ with object sets $\colorsa$ and $\colorsb$, respectively.
Suppose that $X$ is a $(\colorsb,\colorsa)$-collection concentrated in arity one, which moreover comes with the structure of an $\mathcal{M}$-$\mathcal{N}$ bimodule.
In this case, we can regard $X$ as an $\mathcal{E}$-functor
\[
	\mathcal{N}^{\oprm} \times \mathcal{M} \to \mathcal{E}
\]
which on objects sends $(b,a)$ to $X(b;a) = X\vprof{b}{a}$.

Suppose that $Y$ is a $(\colorsc, \colorsb)$-collection which is a left $\mathcal{N}$-module.
Concretely, $X \cp_{\mathcal{N}} Y$ (from Notation~\ref{notation: cp sub Q}) is given by the coend
\[
	(X \cp_{\mathcal{N}} Y)\vprof{\uc}{a} = \int^{b\in \mathcal{N}} X\vprof{b}{a} \otimes Y\vprof{\uc}{b}
\]
and the left $\mathcal{M}$-action is given by
\[ \begin{tikzcd}[sep=small]
 \mathcal{M}\vprof{a}{a'} \otimes \displaystyle \int^{b\in \mathcal{N}} X\vprof{b}{a} \otimes Y\vprof{\uc}{b}   \\
\displaystyle \int^{b\in \mathcal{N}} \mathcal{M}\vprof{a}{a'} \otimes X\vprof{b}{a} \otimes Y\vprof{\uc}{b}  \rar \uar{\cong} &  \displaystyle \int^{b\in \mathcal{N}} X\vprof{b}{a'} \otimes Y\vprof{\uc}{b}. 
\end{tikzcd} \]
Our primary focus going forward will be the case where $\colorsc$ is the empty set.
\begin{lemma}
\label{lemma: underlying adjunction}
Let $\mathcal{M}$ and $\mathcal{N}$ be $\mathcal{E}$-enriched categories with objects $\colorsa$ and $\colorsb$ respectively. 
Let $X$ be a $(\colorsb,\colorsa)$-collection concentrated in arity one, equipped with the structure of an $\mathcal{M}\text{-}\mathcal{N}$ bimodule.
Then the functor
\[
	X\cp_{\mathcal{N}} (-) : \algebras{\mathcal{N}} \to  \algebras{\mathcal{M}}
\]
admits a right adjoint.
\end{lemma}

\begin{notation}
The right adjoint is called $\threehoms{\mathcal{M}}{X}{-}^{\mathcal{M}}.$ 
As the category $\mathcal{M}$ will always be clear from context, we simply use the shorthand $\threehoms{\mathcal{M}}{X}{-}$ for this functor.
That is, the adjunction from Lemma~\ref{lemma: underlying adjunction} takes the form
\[
\begin{tikzcd}
\algebras{\mathcal{N}}
 \ar[rr,"{X\cp_{\mathcal{N}} -}", shift left=2.5]
&\bot &\algebras{\mathcal{M}}.
 \ar[ll,"{\threehoms{\mathcal{M}}{X}{-}}", shift left=2.5]
\end{tikzcd}
\]
\end{notation}

\begin{proof}[Proof of Lemma~\ref{lemma: underlying adjunction}]
This is a categorification of a standard base-change adjunction and follows from the adjunction of the closed monoidal structure on $\mathcal{E}$.
In more detail, suppose that $Z$ is an $\mathcal{M}$-algebra.
Define the underlying $\colorsb$-object of $\threehoms{\mathcal{M}}{X}{Z}$ to be the equalizer of the diagram
\[ \begin{tikzcd}[column sep=tiny, row sep=small]
& \langle X, \langle \mathcal{M}, Z \rangle \rangle \arrow[dr,"\cong"] \\
\langle X, Z \rangle \arrow[ur] \arrow[rr] & & \langle  \mathcal{M} \cp X, Z \rangle \end{tikzcd} \]
where $Z \to \langle \mathcal{M}, Z \rangle$ is the adjoint, from Lemma~\ref{one right adjoint}, of the structure map $\mathcal{M} \cp Z \to Z$.

Let us now turn to the $\mathcal{N}$-action on this $\colorsb$-object.
To produce the map $\mathcal{N} \cp \threehoms{\mathcal{M}}{X}{Z} \to \threehoms{\mathcal{M}}{X}{Z}$ we use the fact that $\mathcal{N} \cp (-)$ is left adjoint to $\langle \mathcal{N}, - \rangle$.
The right adjoint $\langle \mathcal{N}, - \rangle$ commutes with the equalizer, and by adjunction we see that it is equivalent to provide a map
\[
\threehoms{\mathcal{M}}{X}{Z} 
\to 
\langle\mathcal{N},\threehoms{\mathcal{M}}{X}{Z}\rangle
\cong 
\threehoms{\mathcal{M}}{X\cp \mathcal{N}}{Z}
\]
which is then supplied by pulling back along the right action $X\cp \mathcal{N}\to X$.
\end{proof}

\begin{remark}
If $(\colorsc, \mathcal{R})$ is an operad, then the proof of Lemma~\ref{lemma: underlying adjunction} can be readily adapted to provide an adjunction between $\bimod{\mathcal{N}}{\mathcal{R}}$ and $\bimod{\mathcal{M}}{\mathcal{R}}$.
The lemma we have stated is just the special case when $\mathcal{R}$ is the initial colored operad (that is, when $\colorsc = \emptyset$).
Unraveling the equalizer description from the proof reveals that the underlying $(\colorsc, \colorsb)$-collection of $\threehoms{\mathcal{M}}{X}{Z}$ may be written in end notation as 
\[
	\threehoms{\mathcal{M}}{X}{Z}\vprof{\uc}{b} = \int_{a\in \mathcal{M}} \twohoms{X\vprof{b}{a}}{Z\vprof{\uc}{a}},
\]
which compares favorably with the formula from Notation~\ref{notation for angle brackets}.
\end{remark}

\begin{corollary}
\label{cor: underlying adjunction}
Let $(f,\phi)$ be a map of colored operads from $\mathcal{P}$ to ${\calQ}$.
There is an adjunction between $\underlying{{\calQ}}$-algebras and $\underlying{\mathcal{P}}$-algebras with left adjoint
\[
\underlying{\phi}^* = \bar{f} \cp \underlying{{\calQ}} \cp_{\underlying{{\calQ}}} (-) \cong \bar{f} \cp (-)
\]
and right adjoint
\[
\underlying{\phi}_* = \threehoms{\underlying{\mathcal{P}}}{\bar{f} \cp \underlying{{\calQ}}}{-}.
\]
\end{corollary}
\begin{proof}
This is an application of Lemma~\ref{lemma: underlying adjunction} and follows by taking $\mathcal{M} = \underlying{\mathcal{P}}$, $\mathcal{N} = \underlying{{\calQ}}$, and $X = \bar{f} \cp \underlying{{\calQ}}$.
\end{proof}

\begin{remark}
Suppose $(f,\phi):\mathcal{P}\to{\calQ}$ is a map of operads and that $\mathcal{B}$ is a ${\calQ}$-algebra.
We can either restrict the action, considering $\mathcal{B}$ as a $\underlying{{\calQ}}$-algebra, and then apply the functor $\underlying{\phi}^*$, or we can apply $\phi^*$ to $\mathcal{B}$ as in Remark~\ref{remark: usual adjunction between algebras} and then restrict to get a $\underlying{\mathcal{P}}$-algebra. These two procedures coincide.
\end{remark}

\begin{remark}[The adjunctions $\phi_!\dashv \phi^*\dashv \phi_*$ in the categorical situation]
\label{remark: categorical string of adjunctions}
Let $(f,\phi)$ be a map of categories from $\mathcal{P}$ to ${\calQ}$, where ${\calQ}$ has color set $\colorsb$.

By Remark~\ref{remark: right module structure}, ${\calQ}\cp f$ has a ${\calQ}\text{-}\mathcal{P}$ bimodule structure. 
The functor $\threehoms{{\calQ}}{{\calQ}\cp f}{-}$ is naturally isomorphic to $\phi^*$. This can be seen at the level of underlying collections by the following formal manipulation which uses Lemma~\ref{lemma color map adjunctions}:
\begin{align*}
\threehoms{{\calQ}}{{\calQ}\cp f}{\mathcal{B}} 
& \cong \threehoms{{\calQ}}{{\calQ}}{\bar{f}\cp \mathcal{B}}
\\&\cong \twohoms{\mathbf{1}_{\colorsb}}{\bar{f}\cp \mathcal{B}}
\\&\cong \bar{f}\cp \mathcal{B}.
\end{align*}
A diagram chase verifies that the $\mathcal{P}$-actions coincide.
This implies that $\phi^*$ has a left adjoint $\phi_!$ (of course, we already knew this by Remark~\ref{remark: usual adjunction between algebras}).

By Lemma~\ref{lemma: left module structure} $\bar{f}\cp {\calQ}$ has a $\mathcal{P}\text{-}{\calQ}$ bimodule structure.
As we have seen, the functor $(\bar{f}\cp {\calQ})\cp_{{\calQ}}(-)$ from ${\calQ}$-algebras to $\mathcal{P}$-algebras \emph{also} has underlying functor $\bar{f}\cp(-)$ and in fact is also naturally isomorphic to $\phi^*$.
This shows via Lemma~\ref{lemma: underlying adjunction} that $\phi^*$ also has a right adjoint $\phi_*$.

This flexibility in presentation allows us to choose whether to present $\phi^*$ as manifestly a left adjoint or manifestly a right adjoint and justifies the existence of adjoints to $\phi^*$ on both sides in the categorical context.
\end{remark}
In the operadic context, we have already presented $\phi^*$ as `manifestly a right adjoint' (in the sense of the preceding remark) in Remark~\ref{remark: usual adjunction between algebras}.
In the coming sections, we will attempt to present $\phi^*$ as `manifestly a left adjoint', which will lead us to a criterion which permits us to upgrade the relevant presentation from the categorical context.

\section{Main theorem}
\label{section: theorem statements}
In this section we state our main results in generality. 
These appear as Theorem~\ref{theorem: necessary} and Theorem~\ref{theorem: sufficient} which are, respectively, generalizations of the necessity and sufficiency portions of Theorem~\ref{theorem: main theorem, monochrome}.
Several examples are given in Section~\ref{section: examples}.
\begin{lemma}
\label{lemma: extension morphism exists}
Suppose $(f,\phi):\mathcal{P}\to {\calQ}$ is a map of colored operads and let $\lambda$ be the induced left action of $\mathcal{P}$ on $\bar{f} \cp {\calQ}$ from Lemma~\ref{lemma: left module structure}.
The composition 
\[
\mathcal{P} \cp \bar{f}\cp  \underlying{{\calQ}} \to \mathcal{P} \cp \bar{f} \cp {\calQ} \xrightarrow{\lambda} \bar{f} \cp {\calQ}
\]
descends to a morphism
\[
\mathcal{P}\cp_{\underlying{\mathcal{P}}}(\bar{f} \cp \underlying{{\calQ}})\to \bar{f} \cp {\calQ}.
\]
\end{lemma}
\begin{proof}
The domain of the desired morphism, as in Notation~\ref{notation: cp sub Q}, is the coequalizer of the diagram
\[
\begin{tikzcd}[column sep=1.8cm]
\mathcal{P}\cp \underlying{\mathcal{P}} \cp \bar{f} \cp \underlying{{\calQ}}
\rar[shift left]{\mu_{\mathcal{P}}\cp \bar{f} \cp \underlying{{\calQ}}}
\rar[shift right][swap]{\mathcal{P}\cp \lambda_{\underlying{\phi}}}
&
\mathcal{P}\cp \bar{f} \cp \underlying{{\calQ}}.
\end{tikzcd}
\]
Descent to the coequalizer is implied by the $\mathcal{P}$-module relation on $\lambda_\phi$ and the fact that $\lambda_{\underlying{\phi}}$ is the restriction of $\lambda_\phi$.
\end{proof}
\begin{definition}
Let $(f,\phi):\mathcal{P}\to {\calQ}$ be a map of colored operads. 
We call the morphism
\[
\mathcal{P}\cp_{\underlying{\mathcal{P}}}(\bar{f} \cp \underlying{{\calQ}})\to \bar{f} \cp {\calQ}
\]
guaranteed by Lemma~\ref{lemma: extension morphism exists} the \emph{extension morphism} of $(f,\phi)$.
We say $(f,\phi)$ \emph{is a categorical extension}  if the extension morphism is an isomorphism. In this case, we call the inverse isomorphism the \emph{factorization isomorphism}. 
\end{definition}
The factorization isomorphism has components (under adjunction) 
\[\phi(b_1,\ldots, b_n;f(a)):{\calQ}(b_1,\ldots,b_n;f(a))\to (f \cp \mathcal{P}\cp_{\underlying{\mathcal{P}}} \bar{f} \cp \underlying{{\calQ}})(b_1,\ldots, b_n;a).
\]

It is a direct verification that the notion of monoidal extension from Definition~\ref{definition: categorical extension, monochrome} is just the special case of categorical extension when $f$ is the identity function on the point.

\begin{remark}[Assumptions on the ground category]\label{remark ground category}
So far in this paper, we have been working with a bicomplete symmetric monoidal closed category $\mathcal{E}$ (that is, $\mathcal{E}$ is a B\'enabou cosmos).
This is strong enough to establish sufficiency in Theorem~\ref{theorem: main theorem, monochrome}.
Our method of proof to establish necessity goes as follows: for each profile $\vprof{\ub}{a}$, we use the fact that $\phi^*$ preserves colimits to show that a certain map in $\mathcal{E}$ is an isomorphism, and then exhibit the extension morphism of $\phi$ at $\vprof{\ub}{a}$ as a summand of this isomorphism.
We cannot always use this to deduce that the extension morphism is an isomorphism, for instance if $\mathcal{E} = \mathsf{Set}\times \mathsf{Set}^{\oprm}$ (equipped with a Chu-type tensor product).
For questions of necessity, we thus assume that the binary coproduct functor
\begin{align*}
	\mathcal{E} \times \mathcal{E} &\to \mathcal{E} \\
	X, Y &\mapsto X \amalg Y
\end{align*}
is a conservative functor.
That is, if $f\amalg g$ is an isomorphism, then $f$ and $g$ are also isomorphisms\footnote{In general, a functor is \emph{conservative} if the only morphisms it sends to isomorphisms are themselves isomorphisms.}.
We will simply say that \emph{coproducts are conservative} whenever this is the case.
See~\cite[Proposition 3.2]{Borger:DRPC} for some equivalent characterizations of this assumption.
\end{remark}

Coproducts are conservative in many commonly used ground categories, including sets, topological spaces, $R$-modules, spectra, and presheaf categories.

\begin{theorem}
\label{theorem: necessary}
Suppose that coproducts are conservative in $\mathcal{E}$.
Let $(f,\phi)$ be a map of operads $\mathcal{P}\to {\calQ}$.
If the restriction functor $\phi^*$ from ${\calQ}$-algebras to $\mathcal{P}$-algebras is a left adjoint, then $(f,\phi)$ is a categorical extension.
\end{theorem}
This theorem will be proven in Section~\ref{section: necessary}. 
For now, we turn to sufficiency of this criterion, giving a construction of what will turn out to be the right adjoint to $\phi^*$.

\begin{construction}
\label{construction:right adjoint}
Let $(f,\phi)$ be a map of colored operads $\mathcal{P}\to{\calQ}$. 
Recall that $\bar{f} \cp \underlying{{\calQ}}$ is a $\underlying{\mathcal{P}}\text{-}\underlying{{\calQ}}$ bimodule.
Then $\threehoms{\underlying{\mathcal{P}}}{\bar{f} \cp \underlying{{\calQ}}}{-}$ is a functor from $\mathcal{P}$-algebras (or $\underlying{\mathcal{P}}$-algebras) to $\underlying{{\calQ}}$-algebras.
Now suppose $(f, \phi)$ is a categorical extension. 
In this case, we equip the functor $\threehoms{\underlying{\mathcal{P}}}{\bar{f} \cp \underlying{{\calQ}}}{-}$ with a natural transformation $\alpha:({\calQ}\cp (-))\circ \threehoms{\underlying{\mathcal{P}}}{\bar{f} \cp \underlying{{\calQ}}}{-}\to \threehoms{\underlying{\mathcal{P}}}{\bar{f} \cp \underlying{{\calQ}}}{-}$ as follows.
\[
\begin{tikzcd}
{\calQ}\cp \threehoms{\underlying{\mathcal{P}}}{\bar{f} \cp \underlying{{\calQ}}}{\mathcal{A}}
\dar["\text{unit of the adjunction of Corollary~\ref{cor: underlying adjunction}}"]
\\ 
\threehoms{\underlying{\mathcal{P}}}{\bar{f} \cp \underlying{{\calQ}}}{\bar{f} \cp {\calQ}\cp \threehoms{\underlying{\mathcal{P}}}{\bar{f} \cp \underlying{{\calQ}}}{\mathcal{A}}}
\dar[color=blue, dashed, bend left,"\color{black}\text{factorization isomorphism}"]
\\
\uar[color=red, dashed, bend left, "\color{black}\cong"]
\threehoms{\underlying{\mathcal{P}}}{\bar{f} \cp \underlying{{\calQ}}}{\mathcal{P}\cp_{\underlying{\mathcal{P}}} \bar{f} \cp \underlying{{\calQ}}\cp \threehoms{\underlying{\mathcal{P}}}{\bar{f} \cp \underlying{{\calQ}}}{\mathcal{A}}}
\dar["\text{evaluation}"]\\
\threehoms{\underlying{\mathcal{P}}}{\bar{f} \cp \underlying{{\calQ}}}{\mathcal{P}\cp_{\underlying{\mathcal{P}}} \mathcal{A}}
\dar["\mathcal{P}\text{-algebra structure on }\mathcal{A}"]\\
\threehoms{\underlying{\mathcal{P}}}{\bar{f} \cp \underlying{{\calQ}}}{\mathcal{A}}.
\end{tikzcd}\]
% `Evaluation' refers to the composition of $\bar{f} \cp \underlying{{\calQ}}\cp \threehoms{\underlying{\mathcal{P}}}{\bar{f} \cp \underlying{{\calQ}}}{-} 
% \to 
% \bar{f} \cp \underlying{{\calQ}}\cp_{\underlying{{\calQ}}} \threehoms{\underlying{\mathcal{P}}}{\bar{f} \cp \underlying{{\calQ}}}{-}$
% with the counit of the adjunction from Corollary~\ref{cor: underlying adjunction}.
\end{construction}
We emphasize that it is not immediate that $\threehoms{\underlying{\mathcal{P}}}{\bar{f} \cp \underlying{{\calQ}}}{\mathcal{A}}$ is even a $\mathcal{Q}$-algebra when equipped with this structure.
We will return to this in Section~\ref{section: sufficient}.
\begin{remark}
\label{remark: adjoint formula for action}
By the adjunction of Corollary~\ref{cor: underlying adjunction}, specifying a morphism from the $\underlying{{\calQ}}$-algebra ${\calQ}\cp \threehoms{\underlying{\mathcal{P}}}{\bar{f} \cp \underlying{{\calQ}}}{\mathcal{A}}$ to the enriched morphism object $\threehoms{\underlying{\mathcal{P}}}{\bar{f} \cp \underlying{{\calQ}}}{\mathcal{A}}$ is equivalent to giving a $\underlying{\mathcal{P}}$-algebra morphism
\[
\bar{f} \cp \underlying{{\calQ}}\cp_{\underlying{{\calQ}}} {\calQ}\cp \threehoms{\underlying{\mathcal{P}}}{\bar{f} \cp \underlying{{\calQ}}}{\mathcal{A}}\to \mathcal{A}.
\]
The formula for the adjoint of $\alpha$ is somewhat cleaner:
\[
\begin{tikzcd}
\bar{f} \cp {\calQ}\cp \threehoms{\underlying{\mathcal{P}}}{\bar{f} \cp \underlying{{\calQ}}}{\mathcal{A}}
\dar[color=blue, dashed, bend left,"\color{black}\text{factorization isomorphism}"]
\\
\uar[color=red, dashed, bend left, "\color{black}\cong"]
\mathcal{P}\cp_{\underlying{\mathcal{P}}} \bar{f} \cp \underlying{{\calQ}}\cp \threehoms{\underlying{\mathcal{P}}}{\bar{f} \cp \underlying{{\calQ}}}{\mathcal{A}}
\dar{\text{evaluation}} 
\\
\mathcal{P}\cp_{\underlying{\mathcal{P}}}\mathcal{A}
\dar["\mathcal{P}\text{-algebra structure on }\mathcal{A}"]\\
\mathcal{A}.
\end{tikzcd}
\]
For the purposes of exposition, let us be even more concrete. 
Suppose $\mathcal{P}\to {\calQ}$ is a monoidal extension of monochrome operads in sets, and that $\mathcal{A}$ is a $\mathcal{P}$-algebra.
Let us describe the recipe for how an operation $\mu\in {\calQ}(n)$ acts on $n$ $\mathcal{P}(1)$-equivariant functions $f_1,\ldots, f_n$, from ${\calQ}(1)$ to $\mathcal{A}$ to give a new such function $f_0$. 
To apply $f_0$ to $q\in {\calQ}(1)$, we use the factorization morphism to write $q\cp \mu$ as $\mu'\times (q_1,\ldots, q_n)$, where $\mu'$ is in $\mathcal{P}(n)$ and $q_i$ is in ${\calQ}(1)$.
Then $f_0(q)= \mu(f_1(q_1),\ldots,f_n(q_n))$.
\end{remark}
\begin{lemma}
\label{lemma: Q-algebra structure and functor}
Suppose $(f,\phi)$ is a categorical extension.
Then Construction~\ref{construction:right adjoint} lifts the functor $\underlying{\phi}_* = \threehoms{\underlying{\mathcal{P}}}{\bar{f} \cp \underlying{{\calQ}}}{-}$ along the restriction functor from ${\calQ}$-algebras to $\underlying{{\calQ}}$-algebras to a functor $\phi_*$ from $\mathcal{P}$-algebras to ${\calQ}$-algebras.
\end{lemma}
Armed with this lemma, we have functors $\phi^*$ and $\phi_*$ between the categories of $\mathcal{P}$-algebras and ${\calQ}$-algebras. 
The functor $\phi^*$ lies over the functor $\underlying{\phi}^* \cong \bar{f} \cp(-)$ between $\underlying{\mathcal{P}}$-algebras and $\underlying{{\calQ}}$-algebras, and the functor $\phi_*$ lies over its right adjoint $\underlying{\phi}_* = \threehoms{\underlying{\mathcal{P}}}{\bar{f} \cp \underlying{{\calQ}}}{-}$.
\begin{theorem}
\label{theorem: sufficient}
Suppose $(f,\phi)$ is a categorical extension between operads $\mathcal{P}$ and ${\calQ}$.
Then the adjunction $\underlying{\phi}^* \dashv  \underlying{\phi}_*$ of Corollary~\ref{cor: underlying adjunction} underlies an adjunction $\phi^*\dashv \phi_*$ between $\mathcal{P}$-algebras and ${\calQ}$-algebras.
\end{theorem}
Lemma~\ref{lemma: Q-algebra structure and functor} and Theorem~\ref{theorem: sufficient} will be proven in Section~\ref{section: sufficient}.

\section{Examples}
\label{section: examples}
In this section we provide a number of examples of maps of operads that satisfy our criterion, along with some non-examples. We arrange these in order of complexity of the colors involved, not independent interest or importance.

When discussing non-existence of right adjoints, we always restrict to ground categories $\mathcal{E}$ with conservative coproducts so that Theorem~\ref{theorem: necessary} holds.

\subsection{One-colored operads}
\label{section: one-colored examples}
The simplest examples are given by monochrome operads. In this case, as in Theorem~\ref{theorem: main theorem, monochrome}, the criterion on an operad map $\phi:\mathcal{P}\to {\calQ}$ is that
\[
\mathcal{P}\cp_{\mathcal{P}(1)}{\calQ}(1)\to{\calQ}
\]
is an isomorphism. 
\begin{example}
\label{example: arity one}
The most trivial case occurs when ${\calQ}(n)$ and $\mathcal{P}(n)$ are the initial object for $n\ne 1$, that is, the case of monoids. 
Then the criterion is automatically satisfied and the right adjoint from ${\calQ}(1)$-modules to $\mathcal{P}(1)$-modules is the classical base-change functor $\threehoms{\mathcal{P}(1)}{{\calQ}(1)}{-}$ with ${\calQ}(1)$ viewed as a $\mathcal{P}(1)\text{-}{\calQ}(1)$ bimodule.
This includes, for example, change of ring functors between modules over rings related by a ring homomorphism.
We have already seen a slightly more general version of this example as Remark~\ref{remark: categorical string of adjunctions}.
\end{example}
Here is a less trivial example in monochrome operads.
\begin{example}
\label{example: G and BV}
The Gerstenhaber and Batalin--Vilkovisky operads (hereafter $\ger$ and $\bv$) are algebraic models for the $E_2$ and framed $E_2$ operads in chain complexes over a field $\mathbf{k}$ of characteristic zero~\cite{Cohen:HCNS,Getzler:BVATDTFT}. There is an inclusion $\iota:\ger\to \bv$. 

It is well-known~\cite[Proof of Proposition 4.8]{Getzler:BVATDTFT} that the Batalin--Vilkovisky operad $\bv$ is isomorphic as a collection to $\ger\cp (\mathbf{k}[\Delta]/\Delta^2)$, where $\Delta$ is a unary, degree $1$ operator. Here $\mathbf{k}[\Delta]/\Delta^2$ is precisely $\bv(1)$, while $\ger(1)$ is just $\mathbf{k}$. 
The isomorphism between the collection underlying $\bv$ and the collection $\ger\cp (\mathbf{k}[\Delta]/\Delta^2)$ is precisely our factorization isomorphism.

What is the form of the right adjoint? According to our formula, for $X$ a Gerstenhaber algebra, the right adjoint $\iota_*(X)$ is $\twohoms{\mathbf{k}[\Delta]/\Delta^2}{X}$, with the internal hom computed in vector spaces. So the elements of degree $j$ in $\iota_*(X)$ are pairs $(x,y)$ where $x$ is in degree $j$ and $y$ in degree $j+1$. The product is given by
\[
(x_1,y_1)(x_2,y_2)=(x_1x_2, \{x_1,x_2\}+y_1x_2+(-1)^{|x_1|}x_1y_2)
\]
(for some appropriately chosen convention for the Gerstenhaber bracket $\{-,-\}$) and the Batalin--Vilkovisky operator is
\[
\Delta(x,y)=(y,0).
\]

This works in other dimensions and in other ground categories. 
For instance, the inclusion $\iota$ of the little $k$-balls operad into the framed little $k$-balls operad (as operads in a convenient, i.e., bicomplete and Cartesian closed, category of topological spaces) is a monoidal extension by $SO(k)$, so the right adjoint of $\iota$ applied to an algebra $X$ over the little $k$-balls operad yields the space of continuous functions $\operatorname{Map}(SO(k),X)$. 
According to our general setup, the action of an element $\rho$ of $SO(k)$ on a map $g:SO(k)\to X$ is then 
\[
\rho(g)(\sigma)=g(\sigma\cdot\rho).
\]
The action of $E_k^{\mathrm{fr}}(n)$ on a tuple of functions $(g_1,\ldots, g_n) \in \operatorname{Map}(SO(k),X)^{\times n}$ is given by the formula in Remark~\ref{remark: adjoint formula for action}.

This example works in various models, for example the classical model where $E_k(1)$ consists of pairs $(c,r)$ in $(\mathbb{R}^k,\mathbb{R}_+)$ with $\|c\|+r\le 1$ with product
\[
(c,r)\cdot (c',r') = (c+rc',rr')
\]
and $E_k^{\mathrm{fr}}(1)$ is triples $(c,r,\rho)$ in $(\mathbb{R}^k,\mathbb{R}_+,SO(k))$ with $\|c\|+r\le 1$ and product
\[
(c,r,\rho)\cdot (c',r',\rho') = (c+\rho(rc'),rr',\rho\rho').
\]
It is arguably easier to see the monoidal extension property if we instead
choose a point-set model where $E_k(1)$ is a single point and $E_k^{\mathrm{fr}}(1)$ is the group $SO(k)$.
See Figure~\ref{figure: extension for framed little disks} for an illustration of this case.

In particular, for $k=2$, the functor $\iota_*$ is the free loop space functor. 

\end{example}

\begin{figure}
\centering
\begin{tikzpicture}
\draw(0,0) circle [radius = 2cm];
\draw[fill](2,0) circle [radius = 3pt];
\node[below] at (0,-2) {$E_2(2)$};
\draw(-1,0) circle [radius = .6cm];
\node at (-1,0) {$1$};
\draw[fill](-.4,0) circle [radius = 3pt];
\draw(.8,-.8) circle [radius = .4cm];
\node at (.8,-.8) {$2$};
\draw[fill](1.2,-.8) circle [radius = 3pt];
\node[right] at (2.1,0) {$\times\left(\frac{\pi}{2},\frac{5\pi}{4}\right)\longrightarrow$};
\node[below] at (3,-2){$SO(2)^{\times 2}$};
\draw[->](3,-2) -- (3,-.3);
\begin{scope}[xshift = 6.4cm]
\draw(0,0) circle [radius = 2cm];
\draw[fill](2,0) circle [radius = 3pt];
\node[below] at (0,-2) {$E_2^{\mathrm{fr}}(2)$};
\draw(-1,0) circle [radius = .6cm];
\node at (-1,0) {$1$};
\draw[fill](-1,.6) circle [radius = 3pt];
\draw(.8,-.8) circle [radius = .4cm];
\node at (.8,-.8) {$2$};
\draw[fill](.51716,-1.08284) circle [radius = 3pt];
\end{scope}
\end{tikzpicture}
\caption{Extension isomorphism for framed little $2$-disks}
\label{figure: extension for framed little disks}
\end{figure}

\begin{example}
\label{example: dg P}
Let $\mathcal{P}$ be an operad in graded $R$-modules. 
For convenience assume that $\mathcal{P}$ is concentrated in arity at least two 
(with some change of notation the example works more generally).
There is an operad $\mathcal{P}_{\mathrm{dg}}$ in graded $R$-modules whose algebras are differential graded $\mathcal{P}$-algebras.
One way to construct the operad $\mathcal{P}_{\mathrm{dg}}$ is by adjoining a new arity one operation $d$ to $\mathcal{P}$ and imposing the relations that $d$ squares to zero and that $d$ is a graded derivation of every homogeneous operation of $\mathcal{P}$.

The ground category here is \emph{not} chain complexes of $R$-modules and the differential $d$ is an operation in the operad, not part of the structure of an object in the ground category. 

We can then recover $\mathcal{P}$ from $\mathcal{P}_{\mathrm{dg}}$ by quotienting away the operation $d$. 
This can be described as a map of operads $(\id,\phi) : \mathcal{P}_{\mathrm{dg}}\to\mathcal{P}$. 
The extension map is of the form
\[
\mathcal{P}_{\mathrm{dg}}\cp_{\underlying{\mathcal{P}_{\mathrm{dg}}}} \underlying{\mathcal{P}} \to \mathcal{P},
\] 
that is:
\[
\mathcal{P}_{\mathrm{dg}}\cp_{R[d]/d^2} R \to \mathcal{P}.
\]
where the action of $d$ in $R[d]/d^2$ on $R$ is trivial.
Then the left hand side here is the quotient by $d$ which we have already said is isomorphic to $\mathcal{P}$.
Thus $\phi$ is a monoidal extension and $\phi^*$ has both adjoints.

The pullback $\phi^*$ acts on a $\mathcal{P}$-algebra $\mathcal{A}$ by adjoining the zero differential.
The left adjoint $\phi_!$ acts on a differential graded $\mathcal{P}$-algebra $\mathcal{B}$ by quotienting $\mathcal{B}$ (viewed as a $\mathcal{P}$-algebra) by the ideal generated by the differential. 
The right adjoint $\phi_*$ takes a differential graded $\mathcal{P}$-algebra $\mathcal{B}$ to the kernel of $d$, which is still a $\mathcal{P}$-algebra essentially by the derivation property.
\end{example}

\begin{remark}
Examples~\ref{example: G and BV} and~\ref{example: dg P} do not meet Ward's more restrictive criteria~\cite[Proposition 7.9]{Ward:6OFGO}. 
One of his criteria in the fixed color situation is that $\underlying{\mathcal{P}}\to\underlying{{\calQ}}$ is an extension by a group action in each color. 
In both of these examples the extension is by non-invertible elements.
\end{remark}

\begin{example}
\label{example phi one isomorphism}
 A class of non-examples consists of maps of operads $\phi$ where $\phi(1):\mathcal{P}(1)\to{\calQ}(1)$ is an isomorphism.
For such operad maps $\phi$, the product $\mathcal{P}\cp_{\mathcal{P}(1)}{\calQ}(1)$ is isomorphic to $\mathcal{P}$ and the canonical map induced by $\phi$ is just $\phi$ under this identification. 
 So a right adjoint to $\phi^*$ exists if and only if $\phi$ is an isomorphism. 
 \end{example}

Thus if we restrict ourselves to operads whose arity one part consists of just the tensor unit, then the right adjoint $\phi_*$ nearly never exists.
This can be observed directly in many instances, for example there is no right adjoint to the functor from commutative algebras to associative algebras.

\subsection{Bijections on colors}
\label{section: bijection on colors examples}
Now we turn to colored operads, and maps of such which are bijections (or the identity) on color sets.
Our main examples of interest involve colored operads whose algebras are various variants of operads, cyclic operads, and modular operads. 
We have placed explanations of the $\mathsf{Set}$-operads 
$\coloredopfont{O}$, $\coloredopfont{C}$, and $\coloredopfont{M}$
and their variants in Appendix~\ref{appendix: examples of colored operads}.
All of the positive examples work for a general $\mathcal{E}$, by base-changing the colored operads in question along the cocontinuous functor $\mathsf{Set} \to \mathcal{E}$ which sends the point to $\mathbf{1}$.

Suppose we are given a map of operads $(f,\phi)$ from $\mathcal{P}$ to ${\calQ}$ where $f$ is the identity map on some color set $\colorsa$. 
Then the collections $f$ and $\bar{f}$ are both canonically isomorphic to $\mathbf{1}_{\colorsa}$ so $\bar{f} \cp (-)$ is canonically isomorphic to the identity. Then suppressing it in the notation, to be a categorical extension means that the extension morphism is an isomorphism:
\begin{equation*}
\mathcal{P}\cp_{\underlying{\mathcal{P}}}\underlying{{\calQ}}\xrightarrow{\cong} {\calQ}.
\end{equation*}
We conclude that Example~\ref{example phi one isomorphism} holds just as well in this setting, that is, if $\phi : \mathcal{P} \to {\calQ}$ is a bijection-on-colors operad map so that $\underlying{\phi} : \underlying{\mathcal{P}} \to \underlying{{\calQ}}$ is an isomorphism, then $\phi$ is a categorical extension if and only if $\phi$ is an isomorphism.

We turn now to a motivating example for this paper, the relation between operads and cyclic operads.
Recall the colored operads $\coloredopfont{O}$ and $\coloredopfont{C}_{\mathrm{GK}}$ from Example~\ref{example operad for operads} and Example~\ref{example cyclic operads} which have as their respective algebras operads and the cyclic operads of Getzler and Kapranov \cite{GetzlerKapranov:COCH}.
Both operads have as color set the positive integers, and the operations in each are certain isomorphism classes of ordered trees with at least one boundary component.
The difference is that $\coloredopfont{C}_{\mathrm{GK}}$ includes all such ordered trees, whereas $\coloredopfont{O}$ only includes the rooted trees (those with a downward flow).
\begin{lemma}\label{lemma O to CGK}
The map $\coloredopfont{O}\to\coloredopfont{C}_{\mathrm{GK}}$ is a categorical extension.
\end{lemma}
\begin{proof}
Suppose that $G$ is an ordered tree with at least one boundary element.
A preimage of $G$ under the extension morphism will be an ordered rooted tree with the same underlying graph and the same orderings on $B(G)$ and $V(G)$.
Moreover, for each $v$, the two orderings on $\nbhd(v)$ will differ by a cyclic permutation.
There is exactly one possibility for such an ordered rooted tree in the preimage of $G$, hence the extension morphism is an isomorphism.
\end{proof}
This gives another construction of the following adjoint.
\begin{corollary}[Templeton~\cite{Templeton:SGO}]
The forgetful functor from cyclic operads to operads has a right adjoint, with formula given in Construction~\ref{construction:right adjoint}.
\end{corollary}
This adjoint was independently constructed by Ward~\cite[\S9.11]{Ward:6OFGO} in the ground category of vector spaces over a field of characteristic zero. See~\cite[Lemma~2.16]{DrummondColeHackney:CFERIQMSAAIIC} for general ground categories. 
A number of variations, such as for non-unital operads and cyclic operads, or for non-symmetric operads and cyclic operads, are also possible. 
These are left to the reader.

\begin{example}[Non-examples]
\label{example non-examples}
Some negative results are included in Table~\ref{table nonexamples}, where the map of operads $\phi:\mathcal{P}\to{\calQ}$ is not a categorical extension and the exceptional right adjoint to $\phi^*$ does not exist.
We do not give full details but in every case it is possible to find a simple example demonstrating that the extension morphism is not an isomorphism. 
As one example for operads and non-symmetric operads, in Figure~\ref{figure: ns operads variant picture} there is no compatible choice of order for the edges adjacent to the bottom vertex which will make this element appear in the image of the extension morphism of $\coloredopfont{O}_{\mathrm{ns}}\to\coloredopfont{O}$.
As another example, no element of $\coloredopfont{M}$ indexed by a graph containing a cycle of length greater than one is in the image of the extension morphism of $\coloredopfont{C} \to \coloredopfont{M}$.

\begin{table}[ht]
\begingroup
\renewcommand*{\arraystretch}{1.5}
\begin{tabular}
{l>{\raggedright}p{3cm}>{\raggedright}p{3cm}>{\raggedright}p{3cm}}
color set 
& $\mathcal{P}$-algebras 
& ${\calQ}$-algebras 
& $\phi^*$
\tabularnewline
\hline
$\mathbb{N}$
&
non-symmetric operads
&
operads
& 
forget symmetric group actions
\tabularnewline
$\mathbb{N}$
&
cyclic operads 
&
modular operads (without genus)
& 
forget contraction operations
\tabularnewline
$\mathbb{N}\times\mathbb{N}$
&
dioperads 
&
properads
& 
forget higher genus composition
\tabularnewline
$\mathbb{N}\times\mathbb{N}$
&
properads
&
wheeled properads
& 
forget wheel contractions
% \tabularnewline
% $\mathbb{N}\times\mathbb{N}$
% &
% dioperads
% &
% wheeled properads
% & 
% forget wheel contractions and higher genus composition
\end{tabular}
\endgroup
\caption{Non-categorical extensions related to generalized operads}\label{table nonexamples}
\end{table}

These non-examples have both unital and non-unital variants and extend to any ground category $\mathcal{E}$ with conservative coproducts as in Remark~\ref{remark ground category} which is not equivalent to the terminal category.
All of the algebras from Table~\ref{table nonexamples} are about monochrome objects, but these work identically when working over a fixed color set (or in the case of the second line, over a fixed involutive color set), as in Example~\ref{example colored variants}.
For instance, a colored version of the first line would indicate that, for a non-empty set $C$, the forgetful functor from $C$-colored operads to non-symmetric $C$-colored operads does not have a right adjoint.

\begin{figure}[ht]
\labellist
\small\hair 2pt
 \pinlabel {$1$} [ ] at 37.5 6
 \pinlabel {$2$} [ ] at 8 61
 \pinlabel {$3$} [ ] at 57 52
 \pinlabel {$4$} [ ] at 38 61
\endlabellist
\centering
\includegraphics[scale=1.0]{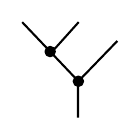}
\caption{No element in the image of $\coloredopfont{O}_{\mathrm{ns}}\to\coloredopfont{O}$ has the indicated boundary ordering.}
\label{figure: ns operads variant picture}
\end{figure}
\end{example}

\subsection{The general case}
\label{section: general case examples}
Now we pass to the general case, where colors are allowed to change.
Let us first consider the cases of the initial and terminal operads.
We already saw in Example~\ref{example phi one isomorphism} that the unique map from a colored operad to the terminal colored operad (the commutative operad) is often not a categorical extension.
\begin{example}[Trivial example]\label{example initial operad trivial}
Suppose that $\mathcal{P}$ is the initial operad, that is, the unique operad with $\colorsa = \emptyset$.
Then for any other operad ${\calQ}$, the unique map $\mathcal{P} \to {\calQ}$ is a categorical extension, as the condition is vacuously satisfied.
The exceptional right adjoint takes the unique object of $\algebras{\mathcal{P}}$ to the terminal object of $\algebras{{\calQ}}$.
\end{example}

Our next example is likely not of independent interest, but addresses a necessary complication of our method of proof (see Remark~\ref{remark on why we can't stick to fixed color case}).

\begin{example}\label{example can't reduce to fixed color}
Let ${\calQ}$ be the $\{1,2\}$-colored operad which is freely generated by a single morphism in ${\calQ}(1,2;1)$.
Let $f: \{1\} \to \{1,2\}$ be the inclusion, and let $\mathcal{P} = \bar{f} \cp {\calQ} \cp f$ be the associated $\{1\}$-colored operad (see Remark~\ref{remark: alt pres col op map}).
That is, $\mathcal{P}$ is the terminal category.
The extension morphism has the form 
\[
	\bar{f} \cp \underlying{{\calQ}} \cong \mathcal{P} \cp_{\underlying{\mathcal{P}}} \bar{f} \cp \underlying{{\calQ}} \to \bar{f} \cp {\calQ}
\]
since $\mathcal{P} = \underlying{\mathcal{P}}$.
As the right-hand side has entries outside of arity one, this map is not an isomorphism.
\end{example}

\begin{remark}\label{remark on why we can't stick to fixed color case}
Every operad map $(f,\phi) : \mathcal{P} \to {\calQ}$ factors as a composite of operad maps
\[
	\mathcal{P} \to \bar{f} \cp {\calQ} \cp f \to {\calQ}
\]
whose first map fixes colors.
If it were the case, for every function $f$ with target $\colorsb$, that $\mathcal{R} = \bar{f} \cp {\calQ} \cp f \to {\calQ}$ happened to be a categorical extension, we would be able to reduce all considerations in the proof of the main theorem to the fixed-color setting.
Alas, Example~\ref{example can't reduce to fixed color} tells us that this is not the case.
Notice, though, that if we apply $(-)\cp f$ to the extension morphism of $\mathcal{R} \to {\calQ}$, we have that
\[
	\mathcal{R} \cp_{\underlying{\mathcal{R}}} \bar{f} \cp \underlying{{\calQ}} \cp f \to \bar{f} \cp {\calQ} \cp f
\]
is always an isomorphism since $\bar{f} \cp \underlying{{\calQ}} \cp f = \underlying{\mathcal{R}}$.
But the functor $(-) \cp f$ is conservative just when $f$ is a surjective function, so we cannot generally deduce in this setting that the extension morphism is an isomorphism.
\end{remark}
Remark~\ref{remark on why we can't stick to fixed color case} yields a reduction to the fixed color case when $f$ is a surjective function. 
There is also a situation which arises in examples in which $f$ is an injective function and $\phi$ is a categorical extension. 
We state it more generally.
\begin{proposition}
Let $(f,\phi)  : \mathcal{P} \to {\calQ}$ be a map of operads such that the comparison map
$\mathcal{P}\cp \bar{f}\to \bar{f}\cp{\calQ}$ is an isomorphism.
Then $\phi$ is a categorical extension.
\end{proposition}
\begin{proof}
The hypothesis implies that $\underlying{\mathcal{P}}\cp \bar{f}\to \bar{f}\cp \underlying{{\calQ}}$ is also an isomorphism.
Consequently
\[
\mathcal{P}\cp_{\underlying{\mathcal{P}}} \bar{f}\cp \underlying{{\calQ}} \cong 
\mathcal{P}\cp_{\underlying{\mathcal{P}}} {\underlying {\mathcal{P}}}\cp \bar{f} \cong
\mathcal{P}\cp \bar{f}\cong
\bar{f}\cp {\calQ}
\]
and this isomorphism coincides with the extension morphism.
\end{proof}

\begin{remark}
\label{remark: injection of color sets and maximal sieve}
When $f:\colorsa\to \colorsb$ is an injection of color sets, the condition that $\mathcal{P}\cp \bar{f}\to \bar{f}\cp {\calQ}$ is an isomorphism reduces to the condition that $\mathcal{P}$ is a \emph{maximal sieve} of ${\calQ}$:
\begin{enumerate}
\item (fully faithful condition) for any profile in $\colorsa$, $\mathcal{P}(\ua;a) \cong {\calQ}(f(\ua);f(a))$ and all structure maps in $\mathcal{P}$ are induced from ${\calQ}$, and
\item (ideal condition) for any profile $(\ub;a)$ in $\listsopb\times \colorsa$, if $\ub$ contains a color not in $\colorsa$ then ${\calQ}(\ub;f(a))$ is an initial object of $\mathcal{E}$.
\end{enumerate}
\end{remark}
This is very close to~\cite[Prop.\ 7.9]{Ward:6OFGO}.
Now we give a few examples of this phenomenon.

\begin{example}
There is a suboperad $\coloredopfont{O}_+$ of $\coloredopfont{O}$ without the color $1$ but with all the same operations for defined profiles. Algebras over $\coloredopfont{O}_+$ correspond to operads with no arity zero operations.

There is in turn a $\{2\}$-colored suboperad $\as$ of $\coloredopfont{O}_{+}$ consisting of those rooted trees where every vertex has valence $2$.
That is, elements of $\as$ are linear rooted trees, including the tree with no vertices. 
Algebras over $\as$ are monoids.
The underlying category of $\as$ has only the identity morphism.
\begin{lemma}
The inclusion of $\as$ into $\coloredopfont{O}_{+}$ is a categorical extension.
\end{lemma}
\begin{proof}
The color map $f$ is the inclusion of $\{2\}$ into $\{2,3,\ldots\}$. 
We use Remark~\ref{remark: injection of color sets and maximal sieve}.
The inclusion of $\as$ into $\coloredopfont{O}_{+}$ satisfies the fully faithful condition by definition.
Any rooted tree with overall valence two and vertices of valence at least two can only have vertices of valence two, so $\as$ satisfies the ideal condition.
\end{proof}

Therefore there is a right adjoint $\phi_*$ to the forgetful functor from operads without $0$-ary operations to monoids. 
By the formula of Construction~\ref{construction:right adjoint}, we have that $\phi_*(A)(n)$ is $A$ if $n$ is $1$ and the terminal object if $n \geq 2$.
This is correct by inspection and is a nice dual result to the shape of the left adjoint, which is the initial object in all arities other than $1$.

Note that the inclusion of $\as$ into the colored operad governing operads with arity zero operations, $\coloredopfont{O}$, is not a categorical extension. 
In this case $\bar{f} \cp \coloredopfont{O}$ contains isomorphism classes of labelled trees with vertices of arbitrary valence and the extension morphism fails to be surjective.
Thus there is no right adjoint to the inclusion of monoids into ordinary operads.
\end{example}

\begin{remark}
The previous example has an extension to the colored setting.
If $\colorsa$ is a set, the operad $\coloredopfont{O}^\colorsa$ for $\colorsa$-colored operads (from Example~\ref{example colored variants}), has analogous suboperads $\as^\colorsa \subseteq \coloredopfont{O}_{+}^\colorsa \subseteq \coloredopfont{O}^\colorsa$.
The inclusion $\as^\colorsa \to \coloredopfont{O}_{+}^\colorsa$ will be a categorical extension, but $\as^\colorsa \to \coloredopfont{O}^\colorsa$ will not be (unless $\colorsa = \varnothing$).
Objects of $\algebras{\as^\colorsa}$ are categories with object set $\colorsa$, and morphisms are identity-on-object functors between such. 
Similarly, $\algebras{\coloredopfont{O}_{+}^\colorsa}$ is the category of $\colorsa$-colored `positive' operads, which play a role in Section~\ref{section: necessary}.
\end{remark}

\begin{example}[$\ell$-truncated operads]
For $1\leq \ell \leq \infty$, let $\coloredopfont{O}_{+}^{\ell}$ denote the full suboperad of $\coloredopfont{O}_{+}$ with color set $\{ 2, 3, \dots, \ell+1 \}$.
Note that $\coloredopfont{O}_{+}^{1} = \as$ and $\coloredopfont{O}_{+}^{\infty} = \coloredopfont{O}_{+}$.
Algebras over $\coloredopfont{O}_{+}^{\ell}$ are `$\ell$-truncated operads.'
If $\ell \leq m$, then the inclusion $\coloredopfont{O}_{+}^{\ell} \to \coloredopfont{O}_{+}^{m}$ is a maximal sieve.
Full faithfulness is automatic, and the ideal condition holds because $\coloredopfont{O}_{+}(k_1, \dots, k_n; p)$ is inhabited if and only if 
\[
	p = 2 + \sum_{i=1}^n (k_i - 2).
\]
By Remark~\ref{remark: injection of color sets and maximal sieve}, $\coloredopfont{O}_{+}^{\ell} \to \coloredopfont{O}_{+}^{m}$ is a categorical extension.
The right adjoint to restriction places the terminal object in arities $\ell < n \leq m$.
When $m=\infty$, this right adjoint $\algebras{\coloredopfont{O}_{+}^{\ell}} \to \algebras{\coloredopfont{O}_{+}}$ appeared as the first part of \cite[Proposition~4.2.2]{GuillenSantosNavarroPascualRoig:MSFO}.
\end{example}

\begin{example}
In Lemma~\ref{lemma O to CGK} we showed that $\coloredopfont{O}\to\coloredopfont{C}_{\mathrm{GK}}$ was a categorical extension, giving a right adjoint to the restriction from Getzler--Kapranov cyclic operads to operads. 
We could just as easily have worked with $\coloredopfont{C}$, the $\mathbb{N}$-colored operad governing all cyclic operads (see Example~\ref{example cyclic operads} for details), as is justified by the following lemma.
\begin{lemma}
The inclusion of operads $\coloredopfont{C}_{\mathrm{GK}}\to\coloredopfont{C}$ is a categorical extension.
\end{lemma}
\begin{proof}
Again we follow Remark~\ref{remark: injection of color sets and maximal sieve}. 
Both full faithfulness and the ideal condition follow from the fact that given a simply connected graph with non-empty boundary, each vertex $v$ of the graph must have non-empty $\nbhd(v)$.
\end{proof}
Therefore the forgetful functor ``forget constants and pairings of elements to constants'' from cyclic operads to Getzler--Kapranov cyclic operads has a right adjoint which puts in a terminal object for the constants and the unique pairing of elements to the terminal object.
In the colored setting, a version of this exceptional right adjoint appears, more or less, as the subcategory inclusion in \cite[Lemma~4.2]{DrummondColeHackney:DKHCO}.
\end{example}
\begin{example}
\label{example: cyclic to NN modular}
We return to the case of cyclic and modular operads.
In Example~\ref{example non-examples}, we noted that $\coloredopfont{C} \to \coloredopfont{M}$ is not a categorical extension.
This time we will use the genus-aware model $\coloredopfont{M}^{\mathrm{g}}$ for modular operads from Example~\ref{example modular genus}. 
There is a map $(f,\phi)$ of colored operads from $\coloredopfont{C}\to \coloredopfont{M}^{\mathrm{g}}$ where $f$ takes the color $n$ to the color $(n,0)$ and $\phi$ takes an ordered tree to itself.
Pullback along this map of colored operads is a forgetful functor from genus-aware modular operads to cyclic operads which forgets all operations and gluings of higher genus.
\begin{lemma}
The map $(f,\phi)$ is a categorical extension.
\end{lemma}
\begin{proof}
In $\coloredopfont{M}^{\mathrm{g}}$, an operation has genus zero inputs and output if and only if the corresponding graph is a tree and the genus of each vertex is zero; these operations are precisely the image of the inclusion of $\coloredopfont{C}$, giving the full faithfulness criterion of Remark~\ref{remark: injection of color sets and maximal sieve}. 
On the other hand, any operation with genus zero output necessarily only has genus zero inputs, giving the ideal criterion.
\end{proof}
We conclude that the ``forget higher genus'' forgetful functor has a right adjoint. Similarly to the previous examples, the adjoint is given by
\[
\phi^*(C)(n,g)=
\begin{cases}
C(n)& g=0
\\{*}&\text{otherwise.}
\end{cases}
 \] 
This adjoint was described by Ward~\cite[\S 9.1]{Ward:6OFGO}.
\end{example}

In Remark~\ref{remark: alt pres col op map} and Remark~\ref{remark on why we can't stick to fixed color case} we encountered a standard construction that allows one to pull back a $\colorsb$-colored operad ${\calQ}$ to an $\colorsa$-colored operad along a function $\colorsa \to \colorsb$.
This can be reinterpreted in terms of the operads appearing in Example~\ref{example colored variants}.

\begin{example}
Let $h : \colorsa \to \colorsb$ be a function.
There is a map of colored operads $(f,\phi) : \coloredopfont{O}^\colorsa \to \coloredopfont{O}^\colorsb$ where $f = \coprod_{n \geq 1} h^{\times n}$ and $\phi$ takes a rooted tree with coloring function $E\to \colorsa$ to the same rooted tree but with coloring function $E\to \colorsa \to \colorsb$.
If $\colorsa$ is empty, then $\coloredopfont{O}^\colorsa$ is the initial operad so Example~\ref{example initial operad trivial} applies to show that $(f,\phi)$ is a categorical extension.
With the assumption that $\colorsa$ is nonempty,
we claim that $(f,\phi)$ is a categorical extension if and only if $h$ is a bijection.
One can see the nontrivial direction by examining the extension morphism in profiles of the form $( (b_1,b_2), (b_2); (a_1))$ where $b_1 = h(a_1)$.
There is a unique rooted ordered tree 
\textbullet\!---\!\textbullet\!--- 
containing two vertices of the appropriate valences.
The target of the extension morphism at this profile is $\coloredopfont{O}^\colorsb( (b_1,b_2), (b_2); (b_1))$, which has exactly one element.
On the other hand, the preimage of this element under the extension morphism may be identified with $h^{-1}(b_2)$, the coloring of the internal edge.
Assuming the extension morphism is an isomorphism, we then have that $h^{-1}(b_2)$ is a single point for every $b_2 \in \colorsb$, hence $h$ is a bijection.
\end{example}

\section{Necessity of the criterion}
\label{section: necessary}
In this section we will prove Theorem~\ref{theorem: necessary}. 
Let $(f,\phi):\mathcal{P}\to{\calQ}$ be a map of colored operads. 
The theorem states that if the restriction functor $\phi^*$ from ${\calQ}$-algebras to $\mathcal{P}$-algebras is a left adjoint, then $(f,\phi)$ is a categorical extension. 
Our strategy is to show that preservation of certain colimits implies the factorization condition.
We begin with initial objects.

\begin{notation}
Let $(\colorsa,\mathcal{P})$ be an operad. 
We write $\initial{\mathcal{P}}$ for the initial $\mathcal{P}$-algebra.
\end{notation}
The initial $\mathcal{P}$-algebra $\initial{\mathcal{P}}$ has $\initial{\mathcal{P}}_a = \mathcal{P}(\,\,; a)$.
This implies the following lemma, whose conclusion is the arity zero part of the condition for a map of operads to be a categorical extension.

\begin{lemma}
Suppose that $(f,\phi):(\colorsa,\mathcal{P})\to(\colorsb,{\calQ})$ is a map of colored operads and that $\phi^*$ preserves initial objects. 
Then for every color $a$ of $\colorsa$, the extension morphism 
\[
\mathcal{P} \cp_{\underlying{\mathcal{P}}} (\bar{f}\cp \underlying{\calQ}) \to \bar{f}\cp \calQ
\]
is an isomorphism in the profile $(\,\,;a)$.
\end{lemma}
\begin{proof}
At profile $(\,\,;a)$, the source of the extension morphism is $\mathcal{P}(\,\,;a)$ while the target is $\calQ(\,\,;f(a))$.
By assumption, $\phi^*\initial{\mathcal{P}} \cong \initialQ$, so we have 
\[
	\mathcal{P}(\,\,;a) = \initial{\mathcal{P}}_a = (\phi^*\initial{\mathcal{P}})_{f(a)} \cong \initialQ_{f(a)} = \calQ(\,\,;f(a)). \qedhere
\]
\end{proof}

Now we will argue that preservation of binary coproducts suffices to show that all of the other components of the extension morphism are isomorphisms.
It is more convenient to assume that we are dealing with \emph{positive} operads. 

\begin{definition}
Let $X$ be an $(\colorsa,\colorsb)$-collection.
We call $X$ \emph{positive} if $X(\,\,; b)$ is the initial object of $\mathcal{E}$ for every color $b$ in $\colorsb$.
Likewise, an operad $(\colorsa, \mathcal{P})$ is positive just when its underlying collection is positive, that is, just when the initial $\mathcal{P}$-algebra coincides with the initial $\colorsa$-object.
\end{definition}

\begin{notation}
Let $\mathcal{P}$ be a colored operad. Write $\mathcal{P}_+$ for the positive operad obtained from $\mathcal{P}$ by replacing the components in profiles with empty input list with the initial object of $\mathcal{E}$.
This procedure is functorial, and if $(f,\phi)$ is a map of colored operads from $\mathcal{P}$ to ${\calQ}$, then we write $(f,\phi_+)$ for the corresponding map of colored operads from $\mathcal{P}_+$ to ${\calQ}_+$.
Further, the natural transformation $\iota$ from $(-)_+$ to the identity functor gives, for each operad $(\colorsa, \mathcal{P})$, a map
\[(\id_\colorsa,\iota_{\mathcal{P}}):\mathcal{P}_+\to \mathcal{P}.\]
\end{notation}
\begin{remark}
\label{remark: P and P+ have similar extension morphisms}
It is reasonable to try to reduce from $\mathcal{P}$ to $\mathcal{P}_+$ because the components of their extension morphisms agree on any profile with non-empty input colors.
That is, the following diagram, with vertical maps extension morphisms and horizontal maps induced by $\iota$, commutes more or less by naturality. 
Moreover, the horizontal maps are isomorphisms in any profile with non-empty input colors:
\[
\begin{tikzcd}
\mathcal{P}_+ \cp_{\underlying{\mathcal{P}_+}} (\bar{f} \cp \underlying{{\calQ}_+})
\dar
\rar
&
\mathcal{P} \cp_{\underlying{\mathcal{P}}} (\bar{f} \cp \underlying{{\calQ}})
\dar
\\
\bar{f}\cp{\calQ}_+\rar 
&
\bar{f}\cp{\calQ}.
\end{tikzcd}
\]
\end{remark}
Coproducts of $\mathcal{P}$-algebras and coproducts of $\mathcal{P}_+$-algebras differ. However, we have the following implication.
\begin{lemma}
\label{lemma: restriction to no arity zero}
Suppose that coproducts are conservative in $\mathcal{E}$ as in Remark~\ref{remark ground category}.
Let $(f,\phi):(\colorsa,\mathcal{P})\to(\colorsb,{\calQ})$ be a map of colored operads. 
If $\phi^*$ preserves initial objects and binary coproducts, then $\phi_+^*$ preserves binary coproducts.
\end{lemma}
We will prove this lemma in Section~\ref{subsec: discarding arity zero}.

Now Theorem~\ref{theorem: necessary} is implied by the following lemma and Remark~\ref{remark: P and P+ have similar extension morphisms}.
\begin{lemma}
\label{lemma: reduced preserves finite coproducts}
Suppose that coproducts are conservative in $\mathcal{E}$.
Let $(f,\phi):(\colorsa,\mathcal{P})\to(\colorsb,{\calQ})$ be a map of positive colored operads.
If $\phi^*$ preserves binary coproducts, then for every positive length list $\ub=(b_1,\ldots, b_n)$ of $\colorsb$ and color $a$ of $\colorsa$, the component of the extension morphism 
\[
\mathcal{P}\cp_{\underlying{\mathcal{P}}}(\bar{f} \cp \underlying{{\calQ}})
(b_1,\ldots, b_n;a)
\to \bar{f} \cp {\calQ}
(b_1,\ldots, b_n;a)
\]
is an isomorphism.
\end{lemma}

The proof of this lemma is the topic of Section~\ref{subsec: preservation of coproducts}.
This proof rests on an analysis of decompositions of collections of the form $X\cp (\coprod Y \cp Z_i)$, which we introduce in Section~\ref{subsec: multilinear}.

\subsection{Discarding arity zero}
\label{subsec: discarding arity zero}
In this section our goal is to prove Lemma~\ref{lemma: restriction to no arity zero}. 

First we observe that restriction along the canonical maps $\iota_{\mathcal{P}}$ and $\iota_{{\calQ}}$ is well-behaved with respect to $\phi$.
\begin{lemma}
\label{lemma: restriction to positive is well-behaved}
Suppose $\mathcal{P}\xrightarrow{(f,\phi)}{\calQ}$ is a map of operads such that $\phi^*$ preserves initial objects.
Then restriction along $\phi$ and $\phi_+$ commutes with induction along $\iota$ in the sense that the following diagram commutes up to natural isomorphism:
\[
\begin{tikzcd}
\algebras{{\calQ}_+}
\rar{(\iota_{{\calQ}})_!}
\dar[swap]{\phi_+^*}
\ar[dr,phantom,"{\rotatebox{45}{$\Rightarrow$}}",description]
&
\algebras{{\calQ}}
\dar{\phi^*}
\\
\algebras{\mathcal{P}_+}
\rar[swap]{(\iota_{\mathcal{P}})_!}
&
\algebras{\mathcal{P}}
\end{tikzcd}
\]
\end{lemma}
\begin{proof}
There is a natural transformation
\begin{equation}\label{eq composite natural transformation}
(\iota_{\mathcal{P}})_!\phi_+^*
\xrightarrow{\text{unit}}
(\iota_{\mathcal{P}})_!\phi_+^* \iota_{{\calQ}}^*(\iota_{{\calQ}})_!
\cong
(\iota_{\mathcal{P}})_! \iota_{\mathcal{P}}^*\phi^*(\iota_{{\calQ}})_!
\xrightarrow{\text{counit}}
\phi^*(\iota_{{\calQ}})_!
\end{equation}
where the isomorphism in the middle follows from $\iota_{{\calQ}}\phi_+=\phi\iota_{\mathcal{P}}$. 
We will argue that this composite natural transformation is an isomorphism, which can be done at the level of underlying collections.
For the moment, we write $U$ for any of the functors which take an algebra to its underlying collection.
At this level, the functors involved have the following form:
\begin{align*}
U\phi_+^* &\cong \bar{f}\cp U(-) & U\phi^* &\cong \bar{f}\cp U(-)
\\
U\iota_{\mathcal{P}}^*&\cong U & U\iota_{{\calQ}}^*&\cong U 
\\
U(\iota_{\mathcal{P}})_!&\cong U\initial{\mathcal{P}}\amalg U(-)
&
U(\iota_{{\calQ}})_!&\cong U\initialQ\amalg U(-).
\end{align*}
Then the morphism underlying the unit natural transformation of ${\calQ}_+$-algebras, $\id\to \iota_{{\calQ}}^*(\iota_{{\calQ}})_!$, of the form
\[
U(-)\to U\initialQ\amalg U(-),
\]
is the universal inclusion of the coproduct.
The morphism underlying the counit of $\mathcal{P}$-algebras, $(\iota_{\mathcal{P}})_!\iota_{\mathcal{P}}^*\to \id$, of the form
\[
U\initial{\mathcal{P}}\amalg U(-)\to U(-),
\]
is given on the first factor by the $\mathcal{P}$-algebra structure and on the second factor by the identity.
That is, for a $\mathcal{P}$-algebra $\mathcal{A}$, the map $U\initial{\mathcal{P}} \to U\mathcal{A}$ underlies the unique $\mathcal{P}$-algebra morphism.
Then the underlying natural transformation of the composite \eqref{eq composite natural transformation} described above, at a ${\calQ}_+$-algebra $\mathcal{B}$, is
\[
\begin{tikzcd}
U\initial{\mathcal{P}}\amalg (\bar{f}\cp U\mathcal{B})\dar{\text{inclusion}}
\\
U\initial{\mathcal{P}}\amalg (\bar{f}\cp(U\initialQ\amalg U\mathcal{B})) \rar{\cong}
& U\initial{\mathcal{P}}\amalg (\bar{f}\cp U\initialQ) \amalg (\bar{f}\cp U\mathcal{B}) \dar
\\
& (\bar{f}\cp U\initialQ) \amalg (\bar{f}\cp U\mathcal{B}).
\end{tikzcd}
\]
Again, the leftmost summand of the second vertical map underlies the unique map of $\mathcal{P}$-algebras $\initial{\mathcal{P}} \to \bar{f} \cp \initialQ = \phi^* \initialQ$, which is an isomorphism by assumption. 
Thus the natural transformation \eqref{eq composite natural transformation} is an isomorphism.
\end{proof}

The functor which takes an algebra to its underlying collection does not preserve coproducts.
At the end of the following proof, and anywhere later in this section where it seems potentially confusing, we will distinguish between coproducts in categories of algebras (over $\mathcal{P}$, $\mathcal{P}_+$, ${\calQ}$, or ${\calQ}_+$) and coproducts in the categories of objects (that is, $\colorsa$ or $\colorsb$-objects) using the notation $\coprod^{\alg}$ and $\coprod^{\obj}$.

\begin{proof}[Proof of Lemma~\ref{lemma: restriction to no arity zero}]
Let $\mathcal{B}_1$ and $\mathcal{B}_2$ be ${\calQ}_+$-algebras.
We want to show that the comparison map
\begin{equation}
\label{eq:desired comparison iso}
\phi_+^*\mathcal{B}_1 \coprod \phi_+^*\mathcal{B}_2 \to \phi_+^*\left(\mathcal{B}_1\coprod \mathcal{B}_2\right)
\end{equation}
is an isomorphism of $\mathcal{P}_+$-algebras.
First, since by hypothesis $\phi^*$ preserves finite coproducts, we know that the following map is an isomorphism of $\mathcal{P}$-algebras:
\[
\phi^*(\iota_{{\calQ}})_!\mathcal{B}_1 \coprod \phi^*(\iota_{{\calQ}})_!\mathcal{B}_2 \xrightarrow{\cong} 
\phi^*\left((\iota_{{\calQ}})_!\mathcal{B}_1\coprod (\iota_{{\calQ}})_!\mathcal{B}_2\right)
\cong
\phi^*(\iota_{{\calQ}})_!\left(\mathcal{B}_1\coprod \mathcal{B}_2\right)
.
\]
By Lemma~\ref{lemma: restriction to positive is well-behaved}, we can replace $\phi^*(\iota_{{\calQ}})_!$ with $(\iota_{\mathcal{P}})_!\phi_+^*$. Then commuting the left adjoint $(\iota_{\mathcal{P}})_!$ past the coproduct, we get
\[
(\iota_{\mathcal{P}})_!\left(\phi_+^*\mathcal{B}_1 \coprod \phi_+^*\mathcal{B}_2\right)
\cong
(\iota_{\mathcal{P}})_!\phi_+^*\mathcal{B}_1 \coprod (\iota_{\mathcal{P}})_!\phi_+^*\mathcal{B}_2 
\xrightarrow {\cong}
(\iota_{\mathcal{P}})_!\phi_+^*\left(\mathcal{B}_1\coprod \mathcal{B}_2\right),
\]
and by inspection this composition is $(\iota_{\mathcal{P}})_!$ applied to~\eqref{eq:desired comparison iso}.

Using the fact that the functor of collections underlying $(\iota_{\mathcal{P}})_!$ is the coproduct with $U\initial{\mathcal{P}}$, the isomorphism of collections underlying our isomorphism of $\mathcal{P}$-algebras is
\[
U\initial{\mathcal{P}}{\coprod}^{\obj} U\left(\phi_+^*\mathcal{B}_1 {\coprod}^{\alg} \phi_+^*\mathcal{B}_2\right)
\xrightarrow{\cong}
U\initial{\mathcal{P}}{\coprod}^{\obj} U\phi_+^*\left(\mathcal{B}_1{\coprod}^{\alg} \mathcal{B}_2\right).
\]
Then since both coproducts and the forgetful functor to collections are conservative,~\eqref{eq:desired comparison iso} is an isomorphism as well.
\end{proof}

\subsection{The multilinear summand}\label{subsec: multilinear}

We will need to be able to refine the description of the composition product of Definition~\ref{definition sub prod} by keeping track of special summands.
In this section, we introduce a decomposition for composition products of a very particular form. 
The key results Lemma~\ref{lemma functoriality of multilinear summand}, Corollary~\ref{corollary: arity one gives ml iso}, and Lemma~\ref{lemma hl summands} will all be needed in Section~\ref{subsec: preservation of coproducts}.

For this section, fix color sets $\colorsa$, $\colorsb$, and $\colorsc$.
Let $X$ be a $(\colorsb,\colorsc)$-collection, let $Y$ be a positive $(\colorsa,\colorsb)$-collection, and for $i\in [k]=\{1,\ldots,k\}$, let $Z_i$ be an $\colorsa$-object.
Assuming that $Y$ is positive and that the $Z_i$ are objects rather than collections are not logically necessary, but this simplification allows us to streamline the definitions and proofs below.

We are interested in summands of the iterated composition products 
\[
X\cp \left(\coprod_{i=1}^k Y\cp Z_i\right)
 \] 
in which each factor $Z_i$ ``appears precisely once''. 
In this section, we will formalize what we mean by this, defining the \emph{multilinear decomposition} of 
$X\cp \left (\coprod Y\cp Z_i\right)$ into \emph{multilinear} and \emph{nonlinear} summands.

Using the point of view of Day powers from Definitions~\ref{definition day powers} and~\ref{definition sub prod}, we will first focus our attention on $Y\cp Z_i$, identifying a coproduct decomposition of the Day power $\left( \coprod Y\cp Z_i\right)^{\ub}$, and then define a coproduct decomposition of the full composition product
$X\cp \left (\coprod Y\cp Z_i\right)$.

The collection $Y$ splits as $Y = Y^{(1)} \amalg Y^{(\geq 2)}$ where $Y^{(1)} = \underlying{Y}$ and $Y^{(\geq 2)}$ has nothing in arity zero or one. 
The decomposition can exclude the 0-ary part of $Y$ because $Y$ is positive.
Let $S$ be the two element set $\{ 1, {\geq}2\}$.
Since $(-) \cp Z_i$ is a left adjoint, we have
\begin{equation} \label{resplitting V}
	\coprod_{i=1}^k Y \cp Z_i = \coprod_{[k] \times S} Y^{(s)} \cp Z_i.
\end{equation}

\begin{definition}
\label{defi: multilinear summand}
Let $Y$ be a positive $(\colorsa,\colorsb)$-collection, and for $i\in [k]=\{1,\ldots,k\}$, let $Z_i$ be an $\colorsa$-object.
Write $V$ for the $\colorsb$-object from \eqref{resplitting V}, and let $\ub$ be a length $m$ list in $\colorsb$.
We have that the power (see Definition~\ref{definition day powers}) is given by
\begin{align*}
V^{\ub} = \bigotimes_{j=1}^m V_{b_j} &= \bigotimes_{j=1}^m \coprod_{[k]\times S} (Y^{(s)} \cp Z_i)_{b_j} \\
&\cong \coprod_{g\times h : [m] \to [k] \times S} \bigotimes_{j=1}^m (Y^{(h(j))} \cp Z_{g(j)})_{b_j}.
\end{align*}
The \emph{multilinear summand} of $V^{\ub}$, denoted by $V^{\ub}_{\ml}$, is the subsum indexed by those $g \times h : [m] \to [k] \times S$ satisfying
\begin{itemize}
 	\item $h(j) = 1$ for all $j$, and
 	\item the map $g:[m]\to[k]$ is a bijection.
   \end{itemize}
In other words, we have
\[
	V^{\ub}_{\ml} \cong \coprod_{g : [m] \cong [k]} \bigotimes_{j=1}^m (\underlying{Y} \cp Z_{g(j)})_{b_j}.
\]
The \emph{nonlinear summand} $V^{\ub}_{\notml}$ consists of the subsum indexed by the remaining choices of $g\times h$.
We have $V^{\ub} = V^{\ub}_{\ml} \amalg V^{\ub}_{\notml}$ and both $V_{\ml}$ and $V_{\notml}$ are functors from $\listsb$ to $\mathcal{E}$.
\end{definition}

\begin{remark}
\label{remark: Day powers multilinear naturality}
Suppose that $Y \to Y'$ is a morphism of positive $(\colorsa,\colorsb)$-collections and, for $i\in [k]$, $Z_i \to Z_i'$ is a map of $\colorsa$-objects.
Further, let $V = \coprod_{i=1}^k Y \cp Z_i$ and $V' = \coprod_{i=1}^k Y' \cp Z_i'$.
Then for every list $\ub$ of elements of $\colorsb$, the induced morphism $V^{\ub} \to (V')^{\ub}$ splits as a sum of $V^{\ub}_{\ml} \to (V')^{\ub}_{\ml}$ and $V^{\ub}_{\notml} \to (V')^{\ub}_{\notml}$.
Further, if each $Z_i \to Z_i'$ is an identity and $Y \to Y'$ is an isomorphism in arity one, then the multilinear summand $V^{\ub}_{\ml} \to (V')^{\ub}_{\ml}$ is an isomorphism.
\end{remark}
This last condition of the remark includes, of course, the inclusion of $\underlying{Y}$ into $Y$.

\begin{definition}
\label{def: multilinear decomp}
Given the multilinear decomposition of the Day powers, we extend to a \emph{multilinear decomposition} of the product $X \cp \coprod_{i=1}^k (Y \cp Z_i)$ as follows
\[
	\left( X \cp \coprod_{i=1}^k (Y \cp Z_i) \right) = \left( X \cp \coprod_{i=1}^k (Y \cp Z_i) \right)_{\ml} 
	\amalg
	\left( X \cp \coprod_{i=1}^k (Y \cp Z_i) \right)_{\notml}
\]
into multilinear and nonlinear summands.
Using the $V$ notation from above, the multilinear summand is given by 
\[
	\left( X \cp \coprod_{i=1}^k (Y \cp Z_i) \right)_{\ml,c} \coloneqq \int^{\ub \in \listsb} X\vprof{\ub}{c} \otimes V^{\ub}_{\ml}
\]
and similarly for the nonlinear summand.
\end{definition}

\begin{definition}
Given two collections equipped with multilinear decompositions, we call a map between them which respects the decompositions \emph{homogeneous}.
\end{definition}

Unraveling the definitions of the summands yields the following lemma.
\begin{lemma}
\label{lemma functoriality of multilinear summand}
Let $X \to X'$ be a map of $(\colorsb,\colorsc)$-collections,
$Y \to Y'$ be a map of positive $(\colorsa,\colorsb)$-collections, and,  
for $i\in [k]=\{1,\ldots,k\}$, let $Z_i \to Z'_i$ be a map of $\colorsa$-objects.
Then the map
\begin{align*}
X \cp \coprod_{i=1}^k (Y \cp Z_i) 
&\to
X' \cp \coprod_{i=1}^k (Y' \cp Z'_i)
\end{align*} 
is homogeneous.
\qed
\end{lemma}
In particular, there are induced maps 
\begin{align*}
\left(X \cp \coprod_{i=1}^k (Y \cp Z_i)\right)_{\ml}
&\to
\left(X' \cp \coprod_{i=1}^k (Y' \cp Z'_i)\right)_{\ml}
\\
\left(X \cp \coprod_{i=1}^k (Y \cp Z_i)\right)_{\notml}
&\to
\left(X' \cp \coprod_{i=1}^k (Y' \cp Z'_i)\right)_{\notml}
\end{align*}
which commute with inclusions of summands. 

Combining this lemma with Remark~\ref{remark: Day powers multilinear naturality} gives the following.
\begin{corollary}
\label{corollary: arity one gives ml iso}
Suppose that $X$ is a $(\colorsb,\colorsc)$-collection, $Y$ is a positive $(\colorsa,\colorsb)$-collection, and, for $i\in [k]=\{1,\ldots,k\}$, $Z_i$ is an $\colorsa$-object.
Then the inclusion $\underlying{Y} \to Y$ induces an isomorphism
\[
\left(X \cp \coprod_{i=1}^k (\underlying{Y} \cp Z_i)\right)_{\ml}
\xrightarrow{\cong}
\left(X \cp \coprod_{i=1}^k (Y \cp Z_i)\right)_{\ml}
\]
of multilinear summands. \qed
\end{corollary}

\begin{lemma}
\label{lemma hl summands}
Let $W$ be a $(\colorsc,\colorsd)$-collection.
Let $X$ be a positive $(\colorsb,\colorsc)$-collection.
Let $Y$ be a positive $(\colorsa,\colorsb)$-collection. 
For $i\in [k]=\{1,\ldots,k\}$, let $Z_i$ be $\colorsa$-objects.
The $\colorsd$-object map
\begin{equation*}
W \cp \coprod_{i=1}^k ((X\cp Y) \cp Z_i) \to W \cp X \cp \coprod_{i=1}^k (Y \cp Z_i)
\end{equation*}
induced by the maps $(X\cp Y) \cp Z_i \cong X \cp (Y\cp Z_i) \to X \cp \coprod_{i=1}^k (Y \cp Z_i)$
is homogeneous.
\end{lemma}

\begin{proof}[Sketch proof]
On the left side, by Remark~\ref{remark: Day powers multilinear naturality}, taking multilinear summands of the relevant Day powers commutes with passing to the arity one part of $X\cp Y$. 
That is,
\[
\left(\coprod_{i=1}^k (X\cp Y)\cp Z_i\right)_{\ml}
\cong
\left(\coprod_{i=1}^k \underlying{X\cp Y}\cp Z_i\right)_{\ml}
\cong
\left(\coprod_{i=1}^k ((\underlying{X}\cp \underlying{Y})\cp Z_i\right)_{\ml},
\]
where the condition to be in the multilinear summand is merely that the function $g$ from Definition~\ref{defi: multilinear summand} is a bijection.
Since $\underlying{X}$ is concentrated in arity one, $\underlying{X}\cp(-)$ distributes over the coproduct by Lemma~\ref{one right adjoint}, and this distribution only changes the function $g$ by reindexing.
This shows that the multilinear summand is preserved by the map in the lemma.
This also shows that if we are in a summand where all $X\cp Y$ factors are in arity one (i.e., we are working with $\underlying{X\cp Y}$) but $g$ is not a bijection, it cannot become a bijection after distribution. 
This is part of the proof that the nonlinear summand is preserved by the map in the lemma.

There is a second way for a summand to be nonlinear, namely if it contains a factor of the form $(X\cp Y)^{(\geq 2)}$. 
Because $X$ and $Y$ are both positive, we have
\[
	(X\cp Y)^{(\geq 2)} = \underlying{X} \cp Y^{(\geq 2)} \amalg X^{(\geq 2)} \cp Y,
\]
so we must have either a factor of $X^{(\geq 2)}$ or a factor of $Y^{(\geq 2)}$.
In the former case, the $g$ on the right side of the distribution cannot be a bijection because some index $i$ in its codomain has multiple preimages. 
In the latter case, the factor from $Y^{(\geq 2)}$ still lives in arity bigger than one after distributing the $X$ factor on the right side and thus will still be nonlinear.
\end{proof}

\subsection{Preservation of coproducts}
\label{subsec: preservation of coproducts}

Now that we have the description of the multilinear summands, we can prove Lemma~\ref{lemma: reduced preserves finite coproducts}.

We will use a description of colimits in algebras over operads using reflexive coequalizers in the ground category due to Rezk~\cite{Rezk:SASCO}.
The only case we will use is the finite coproduct of algebras which is realized by the following reflexive coequalizer.

\begin{proposition}
\label{prop: Rezk formula for finite coproducts}
For a finite coproduct of $\mathcal{P}$-algebras $\mathcal{A}_1$ through $\mathcal{A}_n$, the coproduct $\coprod^{\alg} \mathcal{A}_i$ has as its underlying object the reflexive coequalizer
\[
\mathcal{P}\cp \left( {\coprod_i}^{\obj} \mathcal{P}\cp \mathcal{A}_i \right)\rightrightarrows \mathcal{P}\cp \left( {\coprod_i}^{\obj} \mathcal{A}_i \right)
\]
where the $\coprod^{\obj}$ coproducts are taken in the ground category and the two maps are
\begin{enumerate}
	\item composition in the operad $\mathcal{P}$ and
	\item the action by the operad $\mathcal{P}$ on the algebras $\mathcal{A}_1$ through $\mathcal{A}_n$. \qed
\end{enumerate}
\end{proposition}

Now we will realize the extension morphism
\[
\mathcal{P}\cp_{\underlying{\mathcal{P}}}(\bar{f} \cp \underlying{{\calQ}})
(b_1,\ldots, b_n;-)
\to \bar{f} \cp {\calQ}
(b_1,\ldots, b_n;-)
\]
as a summand of the comparison isomorphism between a coproduct of restrictions and a restriction of a coproduct $\coprod^{\alg}_i \phi^* \mathcal{A}_i \to \phi^* \coprod^{\alg}_i \mathcal{A}_i$.

\begin{notation}
\label{notation: vecb}
Let $(f,\phi):(\colorsa, \mathcal{P})\to (\colorsb,{\calQ})$ be a map of positive colored operads.
Fix a positive length list $\ub$ of $\colorsb$ and a color $a$ of $\colorsa$. 
By abuse of notation write $b_i$ for the $(\emptyset,\colorsb)$-collection which is the unit in the ground category concentrated in profile $(\,\,;b_i)$. 
Write $\freeq$ from $\colorsb$-objects to ${\calQ}$-algebras for the left adjoint to the forgetful functor.
\end{notation}

\begin{lemma}[Homogeneity of the comparison map]
\label{lemma: comparison map is homogeneous}
Let $(f,\phi):\mathcal{P}\to {\calQ}$ be a map of positive colored operads, and let $\ub$ be a positive length list in $\colorsb$.
\begin{enumerate}
\item 
\label{homogeneity lemma item: diagram for coprod of pullback}
The object underlying the coproduct of pulled back algebras $\coprod^{\alg}\phi^*\freeq(b_i)$ can be realized as the coequalizer of a homogeneous diagram
\begin{equation*}
\label{formula for coproduct of pullback of free}
\mathcal{P}\cp \left( {\coprod_i}^{\obj} (\mathcal{P}\cp \bar{f}\cp {\calQ})\cp b_i \right)\rightrightarrows \mathcal{P}\cp \left( {\coprod_i}^{\obj} (\bar{f}\cp{\calQ})\cp b_i \right),
\end{equation*}
with maps induced by 
\begin{itemize}
	\item the action of $\mathcal{P}$ on $\bar{f}\cp{\calQ}$ using $\phi$ and $\mu_{{\calQ}}$ (Lemma~\ref{lemma: left module structure}) and
	\item the distributor for the coproduct followed by the composition $\mu_{\mathcal{P}}$.
\end{itemize}
\item
\label{homegeneity lemma item: diagram for pullback of coprod}
The object underlying the pullback of the coproduct algebra $\phi^*\coprod^{\alg}\freeq(b_i)$
can be realized as the coequalizer of a homogeneous diagram
\begin{equation*}
\label{formula for pullback of coproduct of free}
\bar{f}\cp {\calQ}\cp \left( {\coprod_i}^{\obj} ({\calQ}\cp {\calQ})\cp b_i \right)\rightrightarrows \bar{f}\cp {\calQ}\cp \left( {\coprod_i}^{\obj} {\calQ}\cp b_i \right),
\end{equation*}
with maps induced by
\begin{itemize}
	\item the operadic composition map $\mu_{{\calQ}}$ and
	\item the distributor for the coproduct followed by the composition $\mu_{{\calQ}}$.
\end{itemize}
\item 
\label{homogeneity lemma item: comparison map}
given these presentations, the map underlying the comparison map
\[
{\coprod}^{\alg}\phi^*\freeq(b_i)\to \phi^*{\coprod}^{\alg}\freeq(b_i)
\]
is induced by a homogeneous map of coequalizer diagrams.
\end{enumerate}
\end{lemma}
\begin{proof}
The presentation as a coequalizer in part \ref{homogeneity lemma item: diagram for coprod of pullback} is a direct application of Proposition~\ref{prop: Rezk formula for finite coproducts}.
Since $\bar{f}\cp -$, the functor underlying $\phi^*$, commutes with colimits, the same proposition
yields the presentation as a coequalizer in part \ref{homegeneity lemma item: diagram for pullback of coprod} as well.

In both cases, homogeneity follows from Lemma~\ref{lemma functoriality of multilinear summand} for the first of the two parallel maps.
For the second parallel map it follows from Lemma~\ref{lemma hl summands} coupled with another application of Lemma~\ref{lemma functoriality of multilinear summand}.

Given these presentations, the comparison map in part \ref{homogeneity lemma item: comparison map} is induced by a map of coequalizer diagrams whose components are in turn induced by repeated use of $\phi:\mathcal{P}\cp \bar{f}\to \bar{f}\cp {\calQ}$ and distributors $\coprod (\bar{f}\cp -)\to \bar{f}\cp \coprod -$. 
Every map involved is homogeneous by Lemma~\ref{lemma functoriality of multilinear summand} or~\ref{lemma hl summands}.
\end{proof}

The preceding lemma identifies a splitting of ${\coprod}^{\alg}\phi^*\freeq(b_i)$ into what would be reasonable to call `multilinear' and `nonlinear' parts, despite being of a different form than our general framework from Definition~\ref{def: multilinear decomp}.
We will use this terminology without further comment, and also refer to the map from Lemma~\ref{lemma: comparison map is homogeneous}\eqref{homogeneity lemma item: comparison map} as being homogeneous (the codomain of the comparison map is isomorphic, as a collection, to $\bar{f} \cp {\calQ} \cp \coprod^{\obj} {b_i}$, and so already has a multilinear decomposition).

The following lemma will provide an identification of certain multilinear summands.
\begin{lemma}
\label{lemma: if concentrated in arity one then the multilinear summand picks out the objects}
Suppose that $Y$ is a positive $(\colorsb,\colorsc)$-collection and $X$ is a $(\colorsc,\colorsd)$-collection. Then for each $d\in \colorsd$ and each list $\ub$ of elements of $\colorsb$, there is an isomorphism
\[X\cp \left(\coprod_{i=1}^k Y\cp b_i\right)_{\ml}\vprof{\,}{d}\cong (X\cp \underlying{Y})\vprof{\ub}{d},\]
natural in $X$ and $Y$.
\end{lemma}
\begin{proof}
By Corollary~\ref{corollary: arity one gives ml iso}, we may replace $Y$ by $\underlying{Y}$ on the left-hand side, so for the remainder of the proof we assume $Y= \underlying{Y}$ is concentrated in arity one.
By definition the two sides of the purported isomorphism are
\[
\int^{\uc \in \listsc} X\vprof{\uc}{d} \otimes V^{\uc}_{\ml}
\quad{}\text{and}\quad{}
\int^{\uc\in \listsc} X\vprof{\uc}{d}\otimes Y^{\uc}(\ub),
\]
respectively (taking $Z_i = b_i$ in the specification of $V^{\uc}_{\ml}$ in Definition~\ref{defi: multilinear summand}),
so it suffices to show that $V^{\uc}_{\ml}\cong Y^{\uc}(\ub)$. 
We have
\[
V^{\uc}_{\ml}
\cong 
\coprod_{g:[m]\xrightarrow{\cong}[k]} \bigotimes_{j=1}^m (Y\cp b_{g(i)})_{c_j}
\cong 
\coprod_{g:[m]\xrightarrow{\cong}[k]}\bigotimes_{j=1}^m Y\vprof{b_{g(i)}}{c_j}
\]
which is isomorphic to $Y^{\uc}(\ub)$ by Lemma~\ref{lemma: Day power when Y is arity one}.
\end{proof}

This lemma enables the following identification.

\begin{lemma}
\label{multilinear summand of comparison is extension}
Let $(f,\phi):\mathcal{P}\to {\calQ}$ be a map of positive colored operads, and let $\ub$ be a positive length list in $\colorsb$ (the colors of $\calQ$).

Via the identifications of Lemmas~\ref{lemma: comparison map is homogeneous} and~\ref{lemma: if concentrated in arity one then the multilinear summand picks out the objects}, 
the value of the multilinear summand of the comparison map 
\[
{\coprod}^{\alg}\phi^*\freeq(b_i)\to \phi^*{\coprod}^{\alg}\freeq(b_i)
\]
in profile $\vprof{\,}{a}$ is naturally isomorphic to the component of the extension morphism
\[
\mathcal{P}\cp_{\underlying{\mathcal{P}}}(\bar{f} \cp \underlying{{\calQ}})
\to
\bar{f} \cp {\calQ} 
\]
in profile $\vprof{\ub}{a}$.
\end{lemma}
\begin{proof}
For the domain, since the coequalizer diagram of Lemma~\ref{lemma: comparison map is homogeneous}\eqref{homogeneity lemma item: diagram for coprod of pullback} is homogeneous, the multilinear decomposition distributes through the coequalizer. 
Then by Lemma~\ref{lemma: if concentrated in arity one then the multilinear summand picks out the objects}, the multilinear summand of the domain in profile $\vprof{}{a}$ is the coequalizer of the diagram
\[
(\mathcal{P}\cp \underlying{\mathcal{P}}\cp \bar{f} \cp \underlying{{\calQ}})
\vprof{\ub}{a}
\rightrightarrows
(\mathcal{P}\cp \bar{f} \cp \underlying{{\calQ}})\vprof{\ub}{a},
\]
with maps induced by the action of $\underlying{\mathcal{P}}$ on $\bar{f}\cp\underlying{{\calQ}}$ and the composition $\mu_{\mathcal{P}}$ of the operad $\mathcal{P}$.
This is the domain of the extension morphism.

We can do a similar computation for the codomain using the coequalizer diagram from Lemma~\ref{lemma: comparison map is homogeneous}\eqref{homegeneity lemma item: diagram for pullback of coprod}. 
By Lemma~\ref{lemma: if concentrated in arity one then the multilinear summand picks out the objects} the multilinear summand in profile $\vprof{}{a}$ is the coequalizer of the diagram
\[
\bar{f}\cp {\calQ}\cp \underlying{{\calQ}}\cp \underlying{{\calQ}}\vprof{\ub}{a} 
\rightrightarrows 
\bar{f}\cp {\calQ}\cp \underlying{{\calQ}}\vprof{\ub}{a},
\]
with maps induced by $\mu_{{\calQ}}$. This coequalizer collapses to $\bar{f}\cp {\calQ}\vprof{\ub}{a}$, the codomain of the extension morphism.

By the naturality condition of Lemma~\ref{lemma: if concentrated in arity one then the multilinear summand picks out the objects}, the map between components of $\mathcal{P}\cp_{\underlying{\mathcal{P}}}(\bar{f} \cp \underlying{{\calQ}})$ and $\bar{f} \cp {\calQ}$ is induced by 
\[
\mathcal{P}\cp \bar{f} \cp \underlying{{\calQ}}
\xrightarrow{\phi}
\bar{f}\cp {\calQ}\cp\underlying{{\calQ}}
\]
followed by the collapse to the coequalizer, which is induced by $\mu_{{\calQ}}$.
This is the extension morphism.
\end{proof}

\begin{proof}[Proof of Lemma~\ref{lemma: reduced preserves finite coproducts}]
Since the presentation of the comparison morphism
\[
{\coprod}^{\alg}\phi^*\freeq(b_i)\to \phi^*{\coprod}^{\alg}\freeq(b_i)
\]
of Lemma~\ref{lemma: comparison map is homogeneous} is homogeneous, it respects the multilinear decomposition of that lemma and induces a map between the multilinear summands which was identified in Lemma~\ref{multilinear summand of comparison is extension} as the extension morphism.

By the hypotheses of Lemma~\ref{lemma: reduced preserves finite coproducts}, the comparison morphism is an isomorphism. 
Since coproducts are conservative, its multilinear summand, the extension morphism, is also an isomorphism. 
\end{proof}

\section{Sufficiency of the criterion}
\label{section: sufficient}
Recall that if $\mathcal{P}\xrightarrow{(f,\phi)}{\calQ}$ is any map of operads and $\mathcal{B}$ is a ${\calQ}$-algebra, then by Remark~\ref{remark: usual adjunction between algebras} we know that $\underlying{\phi}^*\mathcal{B} \cong \bar{f} \cp \mathcal{B}$ is not just a $\underlying{\mathcal{P}}$-algebra, but is actually a $\mathcal{P}$-algebra.
By Corollary~\ref{cor: underlying adjunction}, $\underlying{\phi}^*$ has a right adjoint $\underlying{\phi}_*$.
It is not always the case that there is a meaningful ${\calQ}$-algebra structure on $\underlying{\phi}_*\mathcal{A}$ when $\mathcal{A}$ is a $\mathcal{P}$-algebra.
Throughout this section, let \[ \mathcal{P}\xrightarrow{(f,\phi)}{\calQ}\] be a map of operads which is a categorical extension.

Our current task is to prove 
\begin{itemize}
\item Lemma~\ref{lemma: Q-algebra structure and functor}, which says that Construction~\ref{construction:right adjoint} gives a functor $\phi_*$ from $\mathcal{P}$-algebras to ${\calQ}$-algebras, lying over the functor $\underlying{\phi}_* = \threehoms{\underlying{\mathcal{P}}}{\bar{f} \cp \underlying{{\calQ}}}{-}$, and 
\item Theorem~\ref{theorem: sufficient}, which says that the adjunction 
\begin{equation*}
\underlying{\phi}^* = \bar{f} \cp \underlying{{\calQ}} \cp_{\underlying{{\calQ}}} (-)
\dashv 
\threehoms{\underlying{\mathcal{P}}}{\bar{f} \cp \underlying{{\calQ}}}{-} = \underlying{\phi}_*.
\end{equation*}
of Corollary~\ref{cor: underlying adjunction} between $\underlying{\mathcal{P}}$ and $\underlying{\mathcal{Q}}$-algebras lifts to an adjunction $\phi^*\dashv\phi_*$ between $\mathcal{P}$ and $\mathcal{Q}$-algebras.
\end{itemize}

\begin{remark}
\label{remark PL LQ thing}
Let \[ L = \underlying{\phi}^* = \bar{f}\cp\underlying{{\calQ}}\cp_{\underlying{{\calQ}}}(-) \cong \bar{f}\cp(-)\] be the functor from $\underlying{{\calQ}}$-algebras to $\underlying{\mathcal{P}}$-algebras, which is left adjoint to 
\[ R = \underlying{\phi}_* = \threehoms{\underlying{\mathcal{P}}}{\bar{f}\cp \underlying{{\calQ}}}{-}. \]
We have a natural transformation
\begin{equation}\label{eq nat trans for algebras}
	\mathcal{P} \cp_{\underlying{\mathcal{P}}} L(-) \Rightarrow L({\calQ}\cp_{\underlying{{\calQ}}}(-))
\end{equation}
which is given by the extension morphism
\[ \begin{tikzcd}
\mathcal{P} \cp_{\underlying{\mathcal{P}}} \bar{f} \cp \underlying{{\calQ}} \cp_{\underlying{{\calQ}}} \mathcal{B} \to \bar{f} \cp {\calQ} \cp_{\underlying{{\calQ}}} \mathcal{B}.
\end{tikzcd} \]
If $(f,\phi)$ is a categorical extension, then \eqref{eq nat trans for algebras} is an isomorphism of functors from $\underlying{{\calQ}}$-algebras to $\underlying{\mathcal{P}}$-algebras.
In particular, these have the same underlying objects.
\end{remark}
\begin{notation}
To make diagrams in the proofs in this section less busy, we omit the symbol $\cp$ (e.g.,\ $\mathcal{P}(-)$ means $\mathcal{P}\cp(-)$), write $\mathcal{P}\shortdot(-)$ for $\mathcal{P}\cp_{\underlying{\mathcal{P}}}(-)$, and write ${\calQ}\shortdot(-)$ for ${\calQ}\cp_{\underlying{{\calQ}}}(-)$.
We also follow the notation in Remark~\ref{remark PL LQ thing}, writing $R$ for the functor $\threehoms{\underlying{\mathcal{P}}}{\bar{f} \cp \underlying{{\calQ}}}{ -} = \underlying{\phi}_*$ and similarly for its left adjoint $L = \underlying{\phi}^*$.
Further, all functors should be interpreted as being evaluated on everything to the right, which we use to omit all parenthesization.
\end{notation}

For example, \eqref{eq nat trans for algebras} would be written, at a ${\calQ}$-algebra $\mathcal{B}$, as $\mathcal{P}\shortdot L\mathcal{B} \to L{\calQ}\shortdot \mathcal{B}$.

\begin{remark}
We omit the detailed verification that the maps we write down descend to the coequalizers made using $\mathcal{P}\shortdot(-)$ and ${\calQ}\shortdot(-)$. This essentially follows from the fact that everything in sight at least respects $\underlying{\mathcal{P}}$-algebra and $\underlying{{\calQ}}$-algebra structures. 
This includes, in particular, the unit $\eta$, the counit $\epsilon$, the operad composition maps $\mu_{\mathcal{P}}$ and $\mu_{{\calQ}}$, and the $\mathcal{P}$-algebra structure map $\lambda$.
\end{remark}

\begin{remark}\label{remark: induced structural maps}
Let us recast certain induced structures using the notation now available to us.
\begin{itemize}
\item Suppose that $\mathcal{B}$ is a ${\calQ}$-algebra.
The $\mathcal{P}$-algebra structure on $L\mathcal{B}$ from Remark~\ref{remark: usual adjunction between algebras} takes the form 
\[
\begin{tikzcd}
	\mathcal{P} L \mathcal{B} \rar & \mathcal{P} \shortdot L \mathcal{B} \rar & L{\calQ}\shortdot \mathcal{B} \rar{L\lambda} & L\mathcal{B}
\end{tikzcd}
\]
where the middle arrow comes from Remark~\ref{remark PL LQ thing}.
\item
Suppose that $\mathcal{A}$ is a $\mathcal{P}$-algebra.
The proposed ${\calQ}$-algebra structure on $R\mathcal{A}$ comes in two pieces.
\begin{itemize}
\item First, we have the composite from Remark~\ref{remark: adjoint formula for action}, which we write as $\hat\alpha$ below.
\[
\begin{tikzcd}[column sep=2.2em]
L{\calQ} R\mathcal{A} \rar & L{\calQ} \shortdot R\mathcal{A} & \mathcal{P} \shortdot L R \mathcal{A} \rar{\mathcal{P}\shortdot\epsilon} 
\arrow[l, dashed, color=red] 
& \mathcal{P} \shortdot \mathcal{A} \rar{\lambda} & \mathcal{A}.
\end{tikzcd}
\]
The dashed red arrow is \eqref{eq nat trans for algebras} from Remark~\ref{remark PL LQ thing}, which is an isomorphism by assumption.
\item Second, we have the adjunct of $\hat\alpha$, which we call $\alpha$, which is given by 
\[ \begin{tikzcd}
{\calQ} R\mathcal{A} \rar{\eta} & RL {\calQ} R\mathcal{A} \rar{R\hat\alpha} & R\mathcal{A}.
\end{tikzcd} \]
This is our proposed action of ${\calQ}$ on $R\mathcal{A}$ from Construction~\ref{construction:right adjoint}.
\end{itemize}
\end{itemize}
\end{remark}

We start by showing this indeed gives a lift of the functor $R = \underlying{\phi}_*$ to algebras. 
\begin{proof}[Proof of Lemma~\ref{lemma: Q-algebra structure and functor}]
We first address objects.
Let $\mathcal{A}$ be a $\mathcal{P}$-algebra.
We wish to show that $\alpha$ constitutes a ${\calQ}$-algebra structure on $R\mathcal{A}$.
The aim is to show that the diagram
\[ \begin{tikzcd}
{\calQ}{\calQ}R\mathcal{A} \rar{{\calQ} \alpha} \dar{\mu R\mathcal{A}} & {\calQ}R\mathcal{A} \dar{\alpha} \\
{\calQ}R\mathcal{A} \rar{\alpha} & R\mathcal{A}
\end{tikzcd} \]
commutes, or by adjointness that the diagram 
\begin{equation}
\label{adjoint diagram} \begin{tikzcd}
L{\calQ}{\calQ}R\mathcal{A} \rar{L{\calQ}\alpha} \dar{L\mu R\mathcal{A}} & L{\calQ}R\mathcal{A} \rar{L\alpha} \arrow[ddr, "\hat\alpha"] & LR\mathcal{A} \arrow[dd, "\epsilon"]\\
L{\calQ}R\mathcal{A} \dar{L\alpha} \arrow[drr, "\hat \alpha"] \\
LR\mathcal{A} \arrow[rr, "\epsilon"] & & \mathcal{A}
\end{tikzcd} \end{equation}
commutes.
We of course only need to show that the inner chamber of this latter diagram commutes.

The maps $\alpha$ and $\hat \alpha$ depend upon the inverse of the extension isomorphism $\mathcal{P} \cp_{\underlying{\mathcal{P}}} \underlying{{\calQ}} \to {\calQ}$.
As in Remark~\ref{remark: induced structural maps}, we will use dashed red arrows 
$\begin{tikzcd}[cramped] {} \rar[dashed, color=red] & {} \end{tikzcd}$
for maps coming from the natural isomorphism in Remark~\ref{remark PL LQ thing}.

For any $\underlying{{\calQ}}$-algebra $\mathcal{B}$, the diagram
\[ \begin{tikzcd}
\mathcal{P} \shortdot \mathcal{P} \shortdot L\mathcal{B}  \arrow[d, color=red, dashed] \rar{\mu_{\mathcal{P}}\shortdot L\mathcal{B}}& \mathcal{P} \shortdot L \mathcal{B} \arrow[dd, color=red, dashed] \\
\mathcal{P} \shortdot L{\calQ} \shortdot \mathcal{B} \arrow[d, color=red, dashed]  \\
L{\calQ} \shortdot {\calQ} \shortdot \mathcal{B}
\rar{L\mu_{{\calQ}}\shortdot \mathcal{B}} 
& L {\calQ} \shortdot \mathcal{B}
\end{tikzcd} \]
commutes, essentially by the second adjoint form of colored operad maps from Remark~\ref{remark: alt pres col op map}.
When $\mathcal{B} = R\mathcal{A}$, this forms the middle chamber in the following commutative diagram
\[ \begin{tikzcd}[column sep=small]
L{\calQ}{\calQ}R\mathcal{A} \arrow[rr,"L\mu R\mathcal{A}"] \dar &[-0.5em] &[+1.7em] L{\calQ}R\mathcal{A} \dar \arrow[drrr, dotted, color=OliveGreen, "\hat \alpha" swap, bend left=15] &[-0.5em] &[+1.3em] \\
L{\calQ}\shortdot{\calQ}R\mathcal{A} \rar & L{\calQ}\shortdot{\calQ}\shortdot R\mathcal{A} \rar{L\mu\shortdot R\mathcal{A}} & L{\calQ}\shortdot R\mathcal{A} & \mathcal{P}\shortdot  L R\mathcal{A} \arrow[l, dashed, color=red]  \rar{\mathcal{P}\shortdot\epsilon} & \mathcal{P}\shortdot\mathcal{A} \rar{\lambda} & \mathcal{A} \\
\mathcal{P}\shortdot L {\calQ}R\mathcal{A} \arrow[u, dashed, color=red] \rar & \mathcal{P} \shortdot L{\calQ}\shortdot R\mathcal{A} \arrow[u, dashed, color=red] &  & \mathcal{P}\shortdot \mathcal{P} \shortdot LR\mathcal{A} \arrow[ll, dashed, color=red] 
\arrow[u, "\mu\shortdot LR\mathcal{A}"] 
\rar{\mathcal{P}\shortdot\mathcal{P}\shortdot\epsilon} & \mathcal{P}\shortdot \mathcal{P}\shortdot \mathcal{A} \uar{\mu \shortdot\mathcal{A}} \rar{\mathcal{P}\shortdot\lambda} & \mathcal{P}\shortdot \mathcal{A} \uar{\lambda}
\end{tikzcd} \]
whose top composite $L{\calQ}{\calQ}R\mathcal{A} \to \mathcal{A}$ is the left-bottom composite of the inner chamber of \eqref{adjoint diagram}.
Here, all unlabeled solid arrows are induced by the relevant structural maps to coequalizer objects.
On the bottom row, two of the squares commute by naturality, while the one on the right commutes since $\lambda$ is an action.

On the other hand, we have a commutative diagram in Figure~\ref{figure: big comm diagram} in which the composite through the top right corner agrees with the composite through the bottom left corner in the previous diagram. 
Most squares commute by naturality, whereas the upper right triangle is a triangular identity for the pair of adjoint functors.
\begin{figure}
\[ \begin{tikzcd}[column sep=small]
L{\calQ}{\calQ}R\mathcal{A} \rar \dar{L{\calQ}\eta} \arrow[ddddd, bend right=65, "L{\calQ}\alpha" swap, near start, dotted, color=OliveGreen] &
L{\calQ}\shortdot{\calQ}R\mathcal{A} \dar{L{\calQ}\shortdot \eta} &
\mathcal{P}\shortdot L{\calQ}R\mathcal{A} \arrow[l, dashed, color=red] \arrow[dr, bend left=20, "=" swap]  \dar{\mathcal{P}\shortdot L\eta}
\\
L{\calQ}RL{\calQ}R\mathcal{A} \rar \dar &
L{\calQ}\shortdot RL{\calQ}R\mathcal{A} \dar &
\mathcal{P}\shortdot  LRL{\calQ}R\mathcal{A} \arrow[l, dashed, color=red] \rar{\mathcal{P}\shortdot\epsilon} \dar &
\mathcal{P}\shortdot L{\calQ}R\mathcal{A} \dar
\\
L{\calQ}RL{\calQ}\shortdot R\mathcal{A} \rar &
L{\calQ}\shortdot RL{\calQ}\shortdot R\mathcal{A} &
\mathcal{P}\shortdot LRL{\calQ}\shortdot R\mathcal{A} \arrow[l, dashed, color=red] \rar{\mathcal{P}\shortdot\epsilon} &
\mathcal{P}\shortdot L{\calQ}\shortdot R\mathcal{A} 
\\
L{\calQ}R\mathcal{P}\shortdot LR\mathcal{A} \arrow[u, dashed, color=red] \rar \dar{L{\calQ}R\mathcal{P}\shortdot\epsilon}&
L{\calQ}\shortdot R\mathcal{P}\shortdot LR\mathcal{A} \arrow[u, dashed, color=red] \dar{L{\calQ}\shortdot R\mathcal{P}\shortdot\epsilon} &
\mathcal{P}\shortdot LR\mathcal{P}\shortdot LR\mathcal{A} \arrow[l, dashed, color=red] \arrow[u, dashed, color=red] \dar{\mathcal{P}\shortdot LR\mathcal{P}\shortdot \epsilon} \rar{\mathcal{P}\shortdot\epsilon}&
\mathcal{P}\shortdot \mathcal{P} \shortdot LR\mathcal{A} \arrow[u, dashed, color=red] \dar{\mathcal{P}\shortdot \mathcal{P}\shortdot \epsilon}
\\
L{\calQ}R\mathcal{P}\shortdot \mathcal{A} \rar \dar{L{\calQ}R\lambda}  &
L{\calQ}\shortdot R\mathcal{P}\shortdot \mathcal{A} \dar{L{\calQ}\shortdot R\lambda} &
\mathcal{P}\shortdot LR \mathcal{P}\shortdot \mathcal{A} \arrow[l, dashed, color=red] \rar{\mathcal{P}\shortdot\epsilon} \dar{\mathcal{P}\shortdot LR\lambda} &
\mathcal{P}\shortdot \mathcal{P}\shortdot \mathcal{A} \dar{\mathcal{P}\shortdot \lambda}
\\
L{\calQ}R\mathcal{A} \rar \arrow[rrrd, bend right=10, "\hat \alpha", dotted, color=OliveGreen] &
L{\calQ}\shortdot R\mathcal{A} &
\mathcal{P} \shortdot LR\mathcal{A} \arrow[l, dashed, color=red] \rar{\mathcal{P}\shortdot\epsilon} &
\mathcal{P}\shortdot \mathcal{A} \dar{\lambda} \\& & & \mathcal{A}
\end{tikzcd} \]
\caption{Part of the proof of Lemma~\ref{lemma: Q-algebra structure and functor}}\label{figure: big comm diagram}
\end{figure}
The left-bottom composite of this diagram is the top-right composite of the inner chamber of \eqref{adjoint diagram}.
Thus \eqref{adjoint diagram} commutes.

Suppose that $g : \mathcal{A} \to \mathcal{A}'$ is a morphism of $\mathcal{P}$-algebras.
We know that $Rg$ is a morphism of $\underlying{{\calQ}}$-algebras, and we wish to show that this is a morphism of ${\calQ}$-algebras.
That is, the diagram on the left below should commute.
\[ \begin{tikzcd}
{\calQ}R\mathcal{A} \rar{{\calQ}Rg}  \dar{\alpha}& {\calQ}R\mathcal{A}' \dar{\alpha'} \\
R\mathcal{A} \rar{Rg} & R\mathcal{A}'
\end{tikzcd} 
\qquad
\begin{tikzcd}
L{\calQ}R\mathcal{A} \rar{L{\calQ}Rg}  \dar{\hat\alpha}& L{\calQ}R\mathcal{A}' \dar{\hat\alpha'} \\
\mathcal{A} \rar{g} & \mathcal{A}'
\end{tikzcd} 
\]
Of course the diagram on the left is adjoint to the diagram on the right.
Expanding out the definitions of $\hat\alpha$ and $\hat\alpha'$, we have the following diagram.
\[
\begin{tikzcd}
L{\calQ}R\mathcal{A} \rar{L{\calQ}Rg}  \dar \arrow[dddd, bend right=65, "\hat\alpha", dotted, color=OliveGreen] & L{\calQ}R\mathcal{A}' \dar \arrow[dddd, bend left=65, "\hat\alpha'" swap, dotted, color=OliveGreen] \\
L{\calQ}\shortdot R\mathcal{A} \rar{L{\calQ}\shortdot Rg}  & L{\calQ}\shortdot R\mathcal{A}' \\
\mathcal{P} \shortdot LR \mathcal{A} \arrow[u,color=red, dashed] \rar{\mathcal{P} \shortdot LRg} \dar{\mathcal{P} \shortdot \epsilon} & \mathcal{P} \shortdot LR \mathcal{A}' \arrow[u,color=red, dashed] \dar{\mathcal{P} \shortdot \epsilon} \\
\mathcal{P} \shortdot \mathcal{A} \rar{\mathcal{P} \shortdot g} \dar{\lambda_{\mathcal{A}}} & \mathcal{P} \shortdot \mathcal{A}' \dar{\lambda_{\mathcal{A}'}} \\
\mathcal{A} \rar{g} & \mathcal{A}'
\end{tikzcd} 
\]
The top three squares commute by naturality, and the bottom square commutes since $g$ is a map of $\mathcal{P}$-algebras.
\end{proof}

Suppose that $\mathcal{A}$ is a $\mathcal{P}$-algebra and $\mathcal{B}$ is a ${\calQ}$-algebra.
We have now established that the maps in Remark~\ref{remark: induced structural maps} indeed give a ${\calQ}$-algebra structure on $R\mathcal{A} = \underlying{\phi}_* \mathcal{A}$ and a $\mathcal{P}$-algebra structure on $L\mathcal{B} = \underlying{\phi}^* \mathcal{B}$.
We have a pair of functors 
\begin{equation*}
\begin{tikzcd}
\algebras{\mathcal{Q}}
 \rar["\phi^*", shift left]
& \algebras{\mathcal{P}}
 \lar["\phi_*", shift left]
\end{tikzcd}
\end{equation*}
and our goal is to show that they are adjoint.
In order to prove Theorem~\ref{theorem: sufficient}, it is enough to show that the unit and counit of the adjunction
\begin{equation*}
\begin{tikzcd}[column sep=tiny]
\algebras{\underlying{{\calQ}}}
 \ar[rr,"{\underlying{\phi}^*}", shift left=2.5]
&\bot &\algebras{\underlying{\mathcal{P}}}.
 \ar[ll,"\underlying{\phi}_*", shift left=2.5]
\end{tikzcd}
\end{equation*}
from Corollary~\ref{cor: underlying adjunction} are compatible with this additional structure.

\begin{lemma}
\label{lemma: counit calculation}
Let $\mathcal{A}$ be a $\mathcal{P}$-algebra, viewed by restriction as a $\underlying{\mathcal{P}}$-algebra.
The counit of the adjunction $\underlying{\phi}^* \dashv \underlying{\phi}_*$ at $\mathcal{A}$ is a map of $\mathcal{P}$-algebras from $\phi^*\phi_*\mathcal{A}$ to $\mathcal{A}$.
\end{lemma}
\begin{proof}
We must show that the diagram
\begin{equation}\label{diagram counit calculation}
\begin{tikzcd}
\mathcal{P} LR \mathcal{A} \rar{\mathcal{P}\epsilon_A} \dar & \mathcal{P} \mathcal{A} \dar{\lambda}  \\
LR \mathcal{A} \rar{\epsilon_A}& \mathcal{A}
\end{tikzcd} \end{equation}
commutes.
The left-hand map utilizes the formula for the action $\mathcal{P}L\mathcal{B} \to \mathcal{B}$, where $\mathcal{B}$ is a ${\calQ}$-algebra.
This, in turn, relies on the ${\calQ}$-action, called $\alpha$, on $R\mathcal{A}$.
Refer to Remark~\ref{remark: induced structural maps} for both of these.

The bottom two squares of the following diagram commute by naturality, the rightmost cell commutes because the coequalizers are well-behaved, the triangle commutes by a triangular identity, and the top map is defined so that the odd-shaped upper chamber commutes.
\[ \begin{tikzcd}
L{\calQ} \shortdot R\mathcal{A} \dar{L\eta} \arrow[dr, "="] 
\arrow[dddd, bend right=65, "L\alpha" swap, near start, dotted, color=OliveGreen]
&& \mathcal{P} LR \mathcal{A} \ar{ll} \arrow[ddl]
\ar{ddd}{\mathcal{P}\epsilon}
\\
LR L{\calQ} \shortdot R\mathcal{A} \rar{\epsilon} & L{\calQ} \shortdot R\mathcal{A}
\\
LR \mathcal{P} \shortdot L R \mathcal{A} \rar{\epsilon} \dar{LR\mathcal{P}\shortdot \epsilon} \arrow[u, dashed, color=red]& \mathcal{P} \shortdot L R \mathcal{A} \dar{\mathcal{P}\shortdot\epsilon} \arrow[u, dashed, color=red] \\
LR \mathcal{P}\shortdot \mathcal{A} \dar{LR\lambda}\rar{\epsilon}
&
\mathcal{P}\shortdot\mathcal{A}\dar{\lambda}
&\mathcal{P}\mathcal{A}\ar{dl}{\lambda}
\\
LR \mathcal{A} \rar{\epsilon} & \mathcal{A}
\end{tikzcd} \]
The composite (upper-right to lower-left corners) from $\mathcal{P} LR \mathcal{A}$ to $LR \mathcal{A}$ is the left-hand map from \eqref{diagram counit calculation}.
Thus commutativity of \eqref{diagram counit calculation} follows from this commutative diagram.
\end{proof}

\begin{lemma}
\label{lem: unit calculation}
Let $\mathcal{B}$ be a ${\calQ}$-algebra, viewed by restriction as a $\underlying{{\calQ}}$-algebra.
The unit of the adjunction $\underlying{\phi}^* \dashv \underlying{\phi}_*$ at $\mathcal{B}$ is a map of ${\calQ}$-algebras from $\mathcal{B}$ to $\phi_*\phi^*\mathcal{B}$.
\end{lemma}
\begin{proof}
As we already know that the unit is a map of $\underlying{{\calQ}}$-algebras, it is sufficient to show that the diagram
\[ \begin{tikzcd}
{\calQ} \shortdot \mathcal{B} \rar{{\calQ} \shortdot \eta}\dar{\lambda} & {\calQ} \shortdot RL \mathcal{B} \dar{\alpha} \\
\mathcal{B} \rar{\eta} & RL \mathcal{B}
\end{tikzcd} \]
commutes, where the map on the right is the induced action coming from the $\mathcal{P}$-action on $L\mathcal{B}$.
By adjointness, this is equivalent to showing that the triangle
\[ \begin{tikzcd}[column sep=large]
L({\calQ} \shortdot \mathcal{B}) \rar{L({\calQ} \shortdot \eta)} \dar{L\lambda_B}&
L({\calQ} \shortdot RL \mathcal{B}) \arrow[dl] \\
L\mathcal{B}
\end{tikzcd} \]
commutes, where the diagonal map is induced from $\hat\alpha$.
Expanding this slightly, we have the following:
\[ \begin{tikzcd}[column sep=large]
L({\calQ} \shortdot \mathcal{B}) \rar{L({\calQ} \shortdot \eta)} 
\arrow[dd, bend right=50, "L\lambda_{\mathcal{B}}" swap]
& L({\calQ} \shortdot RL \mathcal{B}) \\
\mathcal{P} \shortdot L\mathcal{B} \rar{\mathcal{P} \shortdot L\eta}  \arrow[u, dashed, color=red] \arrow[dr, "="] & 
\mathcal{P} \shortdot LRL\mathcal{B} \arrow[u, dashed, color=red] \dar{\mathcal{P} \shortdot \epsilon} \\
L\mathcal{B} & 
P\shortdot L \mathcal{B} \lar{\lambda_{L\mathcal{B}}}
\end{tikzcd} \]
The bottom left chamber is just the definition of the $\mathcal{P}$-action on $L\mathcal{B}$, while the other two chambers commute automatically.
The composition from upper right to bottom left is induced from $\hat\alpha$, so we have shown that the triangle we want to commute does commute.
\end{proof}
\begin{proof}[Proof of Theorem~\ref{theorem: sufficient}]
In light of Lemma~\ref{lemma: counit calculation} and Lemma~\ref{lem: unit calculation}, we know that the unit and counit of the adjunction $\underlying{\phi}^* \dashv \underlying{\phi}_*$ lift to $\algebras{\mathcal{Q}}$ and $\algebras{\mathcal{P}}$.
Since the functors $\algebras{{\calQ}} \to \algebras{\underlying{{\calQ}}}$ and $\algebras{\mathcal{P}} \to \algebras{\underlying{\mathcal{P}}}$ are faithful, naturality and the triangle identities follow from the corresponding properties for $\underlying{\phi}^* \dashv \underlying{\phi}_*$.
\end{proof}
\begin{remark}
This proof, at the current (colored and categorical) level of generality, is a bit abstract. 
It is a diverting exercise to verify the triangle identities by hand in, say, the case of monochrome operads in sets.
\end{remark}

\appendix
\section{Examples of colored operads}
\label{appendix: examples of colored operads}
Colored operads which describe various types of generalized operads are built on the notion of \emph{graphs with loose ends}, which we will just call \emph{graphs} in what follows.
In such graphs, edges need not be attached to anything at one or both edges, or may be attached to themselves, forming a circle.
A picture is instructive, and we have included one in Figure~\ref{figure: loose ends}.
Any such graph consists of
\begin{itemize}
	\item A finite set of vertices $V$.
	\item A finite set of edges $E$.
	\item For each vertex $v$, a set $\nbhd(v)$ of germs of edges at the vertex; there is a function $\coprod_{v\in V} \nbhd(v) \to E$ whose fibers have cardinality less than three.
	\item A boundary set $B$, equipped with a function $B \to E$ whose fibers have cardinality less than three.
\end{itemize}
The only additional condition to be a graph is that each fiber of $B \amalg \coprod_{v\in V} \nbhd(v) \to E$ has cardinality either zero or two.
If $e$ is an edge such that this fiber is empty, we regard $e$ as being like a circle, while if two elements of $B$ map to $e$, we regard $e$ as being like an interval.
Alternative presentations of such graphs may be found in \cite[Definition~13.1]{BataninBerger:HTAOPM}, \cite[Definition~1.1]{HackneyRobertsonYau:MONT}, and \cite[Section~1.2]{YauJohnson:FPAM}.
There is a realization functor from graphs to topological spaces whose details we omit; when we refer to topological properties of a graph we always implicitly use this functor.
In what follows we always restrict to connected graphs.
An \emph{ordered} graph is a graph $G$ with the following additional structure:
\begin{itemize}
	\item A total ordering on the set of vertices $V$.
	\item A total ordering on each $\nbhd(v)$.
	\item A total ordering on the boundary $B$.
\end{itemize}

\begin{figure}
\includegraphics{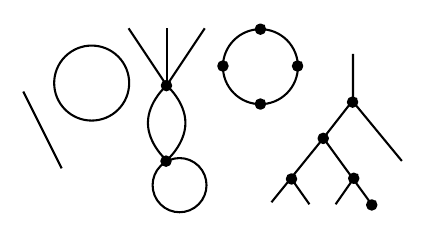}
\caption{A sample (non-connected) graph with loose ends and ten-element boundary set.}
\label{figure: loose ends}
\end{figure}

The following is a special case of \cite[\S4.5]{Raynor:CSMPL}, when $C$ is a point.
Raynor's term `CSM' refers to what we call `modular operad' in this paper, whereas what Raynor calls a `modular operad' we would call a `non-unital modular operad'.

\begin{example}[Modular operads]\label{example modular}
Let $\coloredopfont{M}$ be the $\mathbb{N}$-colored collection whose elements are isomorphism classes of ordered graphs.
Specifically, an element of $\coloredopfont{M}(k_1, \dots, k_n; p)$ will be represented by an ordered graph $G$ with $n$ vertices $\{v_1, \dots, v_n\}$ so that $\nbhd(v_j)$ has cardinality $k_j$ and $B(G)$ has cardinality $p$.
There is a function 
\[
	\coloredopfont{M}(k_1, \dots, k_n; p) \times \prod_{j=1}^n \coloredopfont{M}(\ell_{j,1}, \dots, \ell_{j,m_j}; k_j) \to \coloredopfont{M}(\ell_{1,1}, \dots, \ell_{n,m_n}; p)
\]
which replaces the $v_j\in V(G)$ with a graph $H_j$ with the gluing specified by the unique ordered bijection $B(H_j) \to \nbhd(v_j)$.
This type of graph substitution is both unital (with respect to corollas) by \cite[Lemma 5.31]{YauJohnson:FPAM} and associative by \cite[Theorem 5.32]{YauJohnson:FPAM}, hence $\coloredopfont{M}$ is a colored operad.
Algebras over $\coloredopfont{M}$ are a kind of modular operad. 

The underlying category of $\coloredopfont{M}$ consists of ordered graphs with precisely one vertex. 
Each edge will either be loop at the vertex, or be connected at one end to $v$.
Then an equivalent presentation of this category has objects $\mathbb{N}$ and morphisms from $k$ to $p$ consisting of the data:
\begin{enumerate}
\item an involution $\iota$ on $[k]$ with precisely $p$ fixed points (the boundary edges) and $\frac{k-p}{2}$ free orbits (the loops), and
\item a bijection between the fixed points of $\iota$ with $[p]$.
\end{enumerate}
In particular the set of morphisms is nonempty if and only if $k\ge p$ and $k\equiv p\pmod 2$.
\end{example}

The following example can be recovered from the Feynman category of \cite[\S2.3.3]{KaufmannWard:FC} using the biequivalence of \cite[Theorem 5.16]{BataninKockWeber:RPSFCO}.

\begin{example}[A genus-aware version]\label{example modular genus}
We also have $\coloredopfont{M}^{\mathrm{g}}$, which is a $\mathbb{N}^2$-colored collection defined as follows.
Let $\coloredopfont{M}^{\mathrm{g}}((k_1,g_1),\ldots, (k_n,g_n);(p,g))$ be the set of isomorphism classes of ordered graphs as before but where the graph $G$ is restricted to have first betti number $g-\sum g_j$. 
That is, operations of $\coloredopfont{M}^{\mathrm{g}}$ are \emph{genus-decorated graphs}.
The composition map defined for $\coloredopfont{M}$ respects genus appropriately, making $\coloredopfont{M}^{\mathrm{g}}$ an operad.
There is a map of operads $\coloredopfont{M}^{\mathrm{g}}$ to $\coloredopfont{M}$ which on color sets is projection on the first factor, $(k,g)\mapsto k$.
In the underlying category of $\coloredopfont{M}^{\mathrm{g}}$, the set of morphisms from $(k,g)$ to $(p,h)$ is given by $\underlying{\coloredopfont{M}}(k;p)$ when $k-p = 2(h-g)$, while in other cases it is empty.
\end{example}

\begin{example}[Cyclic operads]\label{example cyclic operads}
Let $\coloredopfont{C}$ be the suboperad of $\coloredopfont{M}$ with the same color set consisting of ordered graphs which are simply-connected.
Algebras over $\coloredopfont{C}$ are a variant of cyclic operads. 
They are slightly more general than the cyclic operads of~\cite{GetzlerKapranov:COCH} because they contain ``constants'' in level $0$ and two ``elements'' in level $1$ can be paired to give a constant. 

We can define a further suboperad $\coloredopfont{C}_{\mathrm{GK}}$ recovering Getzler and Kapranov's cyclic operads precisely. This suboperad has colors the positive integers, and its elements have the additional restriction that the boundary set $B$ is nonempty.
A version of $\coloredopfont{C}_{\mathrm{GK}}$ appeared in \cite[\S1.6.4]{Lukacs:CODSHC}.

In both cases, the underlying category is a disjoint union of symmetric groups. 
That is, the morphisms between different colors are empty, while the endomorphisms of $n$ are the symmetric group $\Sigma_n$.
\end{example}

\begin{example}[Operads]\label{example operad for operads}
We further restrict $\coloredopfont{C}_{\mathrm{GK}}\subseteq \coloredopfont{M}$ to give an operad $\coloredopfont{O}$ governing monochrome operads. 
This is a variant of the description of~\cite[\S1.5.6]{BergerMoerdijk:RCORHA}. 
Our presentation is slightly more complicated but has the virtue of having a direct relationship with $\coloredopfont{C}$ and $\coloredopfont{M}$.
See also \cite[\S1.2]{DehlingVallette:SHTO} and \cite[\S14.1]{YauJohnson:FPAM}. 

Suppose that $G$ is an ordered graph in $\coloredopfont{C}_{\mathrm{GK}}$, that is, suppose that $G$ is a tree with at least one boundary element.
There is a unique edge flow in the direction of the first element of $B(G)$, which we call the root.
That is, we have a partial order with the root as the minimal element.
This allows us to declare that the root of a vertex $v$ is the element of $\nbhd(v)$ that is nearest to the global root.
We call $G$ a \emph{rooted tree} just when, for each $v$, the root of $v$ is also the minimal element of $\nbhd(v)$.
We declare that $\coloredopfont{O} \subseteq \coloredopfont{C}_{\mathrm{GK}}$ is the collection of all rooted trees. Algebras over $\coloredopfont{O}$ are operads. 

The underlying category is again a disjoint union of symmetric groups. But in the underlying category of $\coloredopfont{O}$, the endomorphisms of $n$ are the symmetric group $\Sigma_{n-1}$ (the root remains fixed).

Let us give a derived example.
There is a suboperad $\coloredopfont{O}_{\mathrm{ns}}$ with colors again the positive integers, but fewer operations in most profiles.
Namely, for a rooted tree to be in $\coloredopfont{O}_{\mathrm{ns}}(k_1,\dots, k_n; p)$, the orderings on $B$ must be compatible with the orderings on each $\nbhd(v)$.
Precisely, suppose that $a_1$ and $a_2$ are elements of $\nbhd(v)$ and $b_1$ and $b_2$ are elements of $B$ so that the image of $b_i$ in $E$ is greater than or equal to the image of $a_i$ in the partial order on $E$.
Compatibility means that if $a_1 < a_2$ in the total ordering on $\nbhd(v)$, then $b_1 < b_2$ in the total ordering on $B$.
Algebras over $\coloredopfont{O}_{\mathrm{ns}}$ are nonsymmetric operads.

The underlying category of $\coloredopfont{O}_{\mathrm{ns}}$ has only the identity in each color because, for a graph with one vertex, the compatibility condition forces the orders on $B \cong E$ and $\nbhd(v) \cong E$ to coincide.
This version was studied by van der Laan~\cite{VanderLaan:CKDSHO}.

Our convention for the colors of $\coloredopfont{O}$ are shifted by one from all conventions in the literature. 
We make this nonstandard choice because we are interested in the comparison with $\coloredopfont{C}$ and $\coloredopfont{M}$ where this shift is natural.
\end{example}
The operad $\coloredopfont{O}$ has operations given by rooted trees.
One can imagine analogous operads whose operations are other kinds of directed graphs and whose algebras are dioperads, properads, wheeled operads, and so on.
A general construction of such operads is included in Section 14.1 of \cite{YauJohnson:FPAM}, so we will omit further details here.

\begin{example}[Colored variants]\label{example colored variants}
Given a set $\colorsa$ of colors, one can form an operad $\coloredopfont{O}^\colorsa$ whose set of colors is $\coprod_{n \geq 1} \colorsa^{\times n}$ and whose operations are isomorphism classes of ordered rooted trees equipped with a function from the set of edges to the set $\colorsa$.
Algebras over this $\coprod_{n \geq 1} \colorsa^{\times n}$-colored operad are precisely $\colorsa$-colored operads.
When $\colorsa$ is a point, one recovers the operad $\coloredopfont{O}$.
The underlying category of $\coloredopfont{O}^\colorsa$ is a groupoid of positive length lists of elements of $\colorsa$.

This same pattern extends in a straightforward way to other types of directed graphs, and actually falls under the general construction of \cite[\S14.1]{YauJohnson:FPAM}.
For operadic structures built on undirected graphs, like cyclic operads and modular operads, one has the flexibility to work with an involutive set of colors $\colorsa$.
The main difference is that the coloring function $E\to \colorsa$ should be replaced with an involutive function from the involutive set of oriented edges to $\colorsa$. See, e.g.,~\cite[\S2]{DrummondColeHackney:DKHCO},~\cite{JoyalKock:FGNTCSM},~\cite[\S4.5]{Raynor:CSMPL}, and~\cite[\S2]{HackneyRobertsonYau:MONT} for implementations of this involutive perspective.
\end{example}

\bibliographystyle{amsalpha}
% \raggedright
\bibliography{references-2018}

\providecommand{\bysame}{\leavevmode\hbox to3em{\hrulefill}\thinspace}
\providecommand{\MR}{\relax\ifhmode\unskip\space\fi MR }
\providecommand{\MRhref}[2]{%
  \href{http://www.ams.org/mathscinet-getitem?mr=#1}{#2}
}
\providecommand{\href}[2]{#2}
\begin{thebibliography}{GSNPR05}

\bibitem[BB17]{BataninBerger:HTAOPM}
M.~A. Batanin and C.~Berger, \emph{Homotopy theory for algebras over polynomial
  monads}, Theory Appl. Categ. \textbf{32} (2017), Paper No. 6, 148--253.

\bibitem[B{\'e}n67]{Benabou:IB}
Jean B{\'e}nabou, \emph{Introduction to bicategories}, Reports of the Midwest
  Category Seminar, Lecture Notes in Math., vol.~47, Springer, 1967, pp.~1--77.

\bibitem[BKW18]{BataninKockWeber:RPSFCO}
Michael Batanin, Joachim Kock, and Mark Weber, \emph{Regular patterns,
  substitudes, {F}eynman categories and operads}, Theory Appl. Categ.
  \textbf{33} (2018), Paper No. 7, 148--192.

\bibitem[BM07]{BergerMoerdijk:RCORHA}
Clemens Berger and Ieke Moerdijk, \emph{Resolution of coloured operads and
  rectification of homotopy algebras}, Categories in algebra, geometry and
  mathematical physics, Contemp. Math., vol. 431, Amer. Math. Soc., Providence,
  RI, 2007, pp.~31--58.

\bibitem[B{\"{o}}r94]{Borger:DRPC}
Reinhard B{\"{o}}rger, \emph{Disjointness and related properties of
  coproducts}, Acta Univ. Carolin. Math. Phys. \textbf{35} (1994), no.~1,
  43--63.

\bibitem[Coh76]{Cohen:HCNS}
Fred Cohen, \emph{The homology of {$C_{n+1}$}-spaces, $n\ge 0$}, The Homology
  of Iterated Loop Spaces, Lecture Notes in Math., vol. 533, Springer Berlin
  Heidelberg, Berlin, Heidelberg, 1976, pp.~207--351.

\bibitem[DCH19]{DrummondColeHackney:CFERIQMSAAIIC}
Gabriel~C. Drummond-Cole and Philip Hackney, \emph{A criterion for existence of
  right-induced model structures}, Bull. Lond. Math. Soc. \textbf{51} (2019),
  no.~2, 309--326.

\bibitem[DCH21]{DrummondColeHackney:DKHCO}
\bysame, \emph{Dwyer--{K}an homotopy theory for cyclic operads}, Proc. Edinb.
  Math. Soc. (2) \textbf{64} (2021), no.~1, 29--58.

\bibitem[DV15]{DehlingVallette:SHTO}
Malte Dehling and Bruno Vallette, \emph{Symmetric homotopy theory of operads},
  \href{https://arxiv.org/abs/1503.02701}{arXiv:1503.02701} [math.AT], to
  appear in Algebr. Geom. Topol., 2015.

\bibitem[Get94]{Getzler:BVATDTFT}
Ezra Getzler, \emph{Batalin--{V}ilkovisky algebras and two-dimensional
  topological field theories}, Comm. Math. Phys. \textbf{159} (1994), 265--285.

\bibitem[GK95]{GetzlerKapranov:COCH}
E.~Getzler and M.~M. Kapranov, \emph{Cyclic operads and cyclic homology},
  Geometry, topology, \& physics, Conf. Proc. Lecture Notes Geom. Topology, IV,
  Int. Press, Cambridge, MA, 1995, pp.~167--201.

\bibitem[GSNPR05]{GuillenSantosNavarroPascualRoig:MSFO}
Francisco Guill\'en~Santos, Vincen{\c c} Navarro, Pere Pascual, and Agust\'i
  Roig, \emph{Moduli spaces and formal operads}, Duke Math. J. \textbf{129}
  (2005), 291--335.

\bibitem[HRY16]{HackneyRobertsonYau:RLPCO}
Philip Hackney, Marcy Robertson, and Donald Yau, \emph{Relative left properness
  of colored operads}, Algebr. Geom. Topol. \textbf{16} (2016), no.~5,
  2691--2714.

\bibitem[HRY20]{HackneyRobertsonYau:MONT}
\bysame, \emph{Modular operads and the nerve theorem}, Adv. Math. \textbf{370}
  (2020), 107206, 39.

\bibitem[JK11]{JoyalKock:FGNTCSM}
Andr\'e Joyal and Joachim Kock, \emph{Feynman graphs, and nerve theorem for
  compact symmetric multicategories (extended abstract)}, Electron. Notes
  Theor. Comput. Sci. \textbf{270} (2011), no.~2, 105--113.

\bibitem[Kel05]{Kelly:OOJPM}
G.~M. Kelly, \emph{On the operads of {J}. {P}. {M}ay}, Repr. Theory Appl.
  Categ. (2005), no.~13, 1--13.

\bibitem[KW17]{KaufmannWard:FC}
Ralph~M. Kaufmann and Benjamin~C. Ward, \emph{Feynman categories},
  Ast\'{e}risque (2017), no.~387, vii+161.

\bibitem[Lei04]{Leinster:HOHC}
Tom Leinster, \emph{Higher operads, higher categories}, London Mathematical
  Society Lecture Note Series, vol. 298, Cambridge University Press, Cambridge,
  2004.

\bibitem[Luk10]{Lukacs:CODSHC}
Andor Luk\'acs, \emph{Cyclic {O}perads, {D}endroidal {S}tructures, {H}igher
  {C}ategories}, Ph.D. thesis, Universiteit Utrecht, 2010.

\bibitem[Ray18]{Raynor:CSMPL}
Sophie Raynor, \emph{Compact symmetric multicategories and the problem of
  loops}, Ph.D. thesis, University of Aberdeen, 2018.

\bibitem[Rez96]{Rezk:SASCO}
Charles Rezk, \emph{Spaces of algebra structures and cohomology of operads},
  Ph.D. thesis, Massachusetts Institute of Technology, 1996.

\bibitem[Smi82]{Smirnov:HTC}
V.A. Smirnov, \emph{On the cochain complex of topological spaces}, Math. USSR
  Sbornik \textbf{43} (1982), 133--144.

\bibitem[Str72]{Street:FTM}
Ross Street, \emph{The formal theory of monads}, J. Pure Appl. Algebra
  \textbf{2} (1972), 149--168.

\bibitem[Tem03]{Templeton:SGO}
Joseph~James Templeton, \emph{Self-dualities, graphs and operads}, Ph.D.
  thesis, University of Cambridge, 2003.

\bibitem[vdL03]{VanderLaan:CKDSHO}
Pepijn van~der Laan, \emph{Coloured {K}oszul duality and strongly homotopy
  operads}, \href{https://arxiv.org/abs/math/0312147v2}{arXiv:math/0312147v2}
  [math.QA], 2003.

\bibitem[War19]{Ward:6OFGO}
Benjamin~C. Ward, \emph{Six operations formalism for generalized operads},
  Theory Appl. Categ. \textbf{34} (2019), Paper No. 6, 121--169.

\bibitem[Wra70]{Wraith:AT}
G.~C. Wraith, \emph{Algebraic theories}, Lectures Autumn 1969. Lecture Notes
  Series, No. 22, Matematisk Institut, Aarhus Universitet, Aarhus, 1970.

\bibitem[Yau19]{Yau:IOMCGE}
Donald Yau, \emph{Infinity operads and monoidal categories with group
  equivariance}, \href{https://arxiv.org/abs/1903.03839v1}{arXiv:1903.03839v1}
  [math.CT], 2019.

\bibitem[YJ15]{YauJohnson:FPAM}
Donald Yau and Mark~W. Johnson, \emph{A {F}oundation for {PROP}s, {A}lgebras,
  and {M}odules}, Mathematical Surveys and Monographs, vol. 203, American
  Mathematical Society, Providence, RI, 2015.

\end{thebibliography}
\end{document}